\documentclass[11pt]{amsart}
\usepackage{amsfonts,latexsym,amsthm,amssymb,graphicx}
\usepackage{mathabx}
\usepackage[all]{xy}
\usepackage[usenames]{color} 
\usepackage{tikz}
\usetikzlibrary{shapes}
\usetikzlibrary{decorations.pathreplacing}

\DeclareFontFamily{OT1}{rsfs}{}
\DeclareFontShape{OT1}{rsfs}{n}{it}{<-> rsfs10}{}
\DeclareMathAlphabet{\mathscr}{OT1}{rsfs}{n}{it}

\CompileMatrices

\newtheorem{theorem}{Theorem}[section]
\newtheorem{lemma}[theorem]{Lemma}
\newtheorem{corol}[theorem]{Corollary}
\newtheorem{prop}[theorem]{Proposition}

{\theoremstyle{definition} \newtheorem{defin}[theorem]{Definition}}
{\theoremstyle{remark} \newtheorem{remark}[theorem]{Remark}
\newtheorem{example}[theorem]{Example}}

\newcommand{\Pbb}{{\mathbb{P}}}
\newcommand{\cE}{{\mathscr E}}
\newcommand{\cL}{{\mathscr L}}
\newcommand{\cO}{{\mathscr O}}
\newcommand{\cQ}{{\mathscr Q}}
\newcommand{\cS}{{\mathscr S}}
\newcommand{\htau}{{\widehat \tau}}
\newcommand{\hL}{{\widehat L}}
\newcommand{\cQd}{{\cQ^\vee}}
\newcommand{\cSd}{{\cS^\vee}}
\newcommand{\vZ}{{\widecheck Z}}
\newcommand{\Til}[1]{{\widetilde{#1}}}

\DeclareMathOperator{\rk}{rk}
\DeclareMathOperator{\codim}{codim}
\DeclareMathOperator{\Hom}{Hom}

\newcommand{\qede}{\hfill$\lrcorner$}

\title{
Degrees of projections of rank loci
}
\author{Paolo Aluffi}
\address{
Mathematics Department, 
Florida State University,
Tallahassee FL 32306, U.S.A.
}
\email{aluffi@math.fsu.edu}

\begin{document}

\begin{abstract}
We provide formulas for the degrees of the projections of the locus of square matrices with
given rank from linear spaces spanned by a choice of matrix entries. The motivation for
these computations stem from applications to `matrix rigidity'; we also view them as an
excellent source of examples to test methods in intersection theory, particularly computations 
of Segre classes. Our results are generally expressed in terms of intersection numbers in
Grassmannians, which can be computed explicitly in many cases. We observe that, surprisingly 
(to us), these degrees appear to match the numbers of {\em Kekul\'e structures\/} 
of certain {\em benzenoid hydrocarbons,\/} and arise in many other contexts with no 
apparent direct connection to the enumerative geometry of rank conditions.
\end{abstract}

\maketitle


\section{Introduction}\label{intro}
Consider the space $\Pbb^{n^2-1}$ of $n\times n$ complex matrices. If $S$ is a set
of $s$ coordinates in this space, i.e., a set of entries of $n\times n$ matrices, we denote
by $L_S$ the subspace of $\Pbb^{n^2-1}$ spanned by $S$, and we consider the
projection $\pi_S: \Pbb^{n^2-1} \dashrightarrow \Pbb^{n^2-s-1}$ centered at $L_S$.
We depict $S$ by highlighting its entries in an $n\times n$ grid:
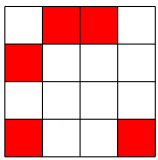
\begin{center}
\begin{tikzpicture}
\path [fill=red] (0.5,0) rectangle (1,-.5);
\path [fill=red] (1,0) rectangle (1.5,-.5);
\path [fill=red] (0,-0.5) rectangle (.5,-1);
\path [fill=red] (0,-1.5) rectangle (.5,-2);
\path [fill=red] (1.5,-1.5) rectangle (2,-2);
\draw (0,0) --(0,-2);
\draw (.5,0) --(.5,-2);
\draw (1,0) --(1,-2);
\draw (1.5,0) --(1.5,-2);
\draw (2,0) --(2,-2);
\draw (0,0) --(2,0);
\draw (0,-.5) --(2,-.5);
\draw (0,-1) --(2,-1);
\draw (0,-1.5) --(2,-1.5);
\draw (0,-2) --(2,-2);
\end{tikzpicture}
\end{center}

The problem we approach in this paper is the computation of the degree of the closure
of the image via $\pi_S$ of the locus $\sigma_{n,r}$ of matrices of rank $\le r$. 

The original motivation for this problem stems from complexity theory.
In~\cite{arXiv1310.1362}, following L.~Valiant (cf.~\cite{valiant}),
the {\em $r$-rigidity\/} of an $n\times n$ matrix $M$ is defined to be the smallest $s$ 
such that there exists a set $S$ of $s$ entries for which $\pi_S(M)\in \overline{\pi_S(\sigma_{n,r})}$. 
There is interest in finding explicit equations for the set of matrices with given rigidity;
these may be used to give lower bounds for the rigidity of a given matrix, with applications
to the complexity of performing explicit operations such as matrix multiplication. We
address the reader to~\cite{arXiv1310.1362} for a thorough discussion of this motivation,
and for a lucid explanation of the algebro-geometric aspects of this question. One idea
introduced in~\cite{arXiv1310.1362} is that the {\em degree\/} $d_{n,r,S}$
of the projection $\overline{\pi_S(\sigma_{n,r})}$ (or equivalently of the join of $L_S$ and 
$\sigma_{n,r}$) gives useful information on polynomial 
equations for matrices with given rigidity. A method for the computation of $d_{n,r,S}$
by successive projections from the entries in $S$ is introduced in~\cite{arXiv1310.1362}
and used to obtain several concrete results. An explicit computation of the value of 
$d_{n,r,S}$ for sets $S$ with no two elements on the same row or column (i.e., up to 
permutations, consisting of diagonal entries) is also presented in~\cite{arXiv1310.1362}, 
subject to a conjectural statement of the behavior of taking joins and tangent spaces in 
the situation arising in this case. See~\S\ref{ss:manyrows2} for a more precise statement
of this conjecture.

Our main purpose in this paper is to illustrate a different method for the computation 
of $d_{n,r,S}$, based on Fulton-MacPherson intersection theory. (The precise definition
we will adopt for $d_{n,r,S}$ is given in \S\ref{ss:dnrS}; it is the degree of 
$\overline{\pi_S(\sigma_{n,r})}$ when the projection is birational to its image.)
The computation is
reduced to the computation of the {\em Segre class\/} of the intersection 
$L_S\cap \sigma_{n,r}$ in $\sigma_{n,r}$, and this in turn may be expressed in terms
of a Segre class involving the standard resolution of~$\sigma_{n,r}$. 
Computing these Segre classes in complete generality appears to be a very challenging 
problem, and it would be worthwhile, for this application as well as other
applications of Segre classes, to develop techniques capable of performing such 
computations. We carry this out for sets~$S$ which can be constructed from `blocks' 
of certain types. For example, we give a formula for the case in which no two elements 
of $S$ lie in the same column
or row of the matrix. In particular, this verifies the formula for diagonal $S$ mentioned 
above, thereby (indirectly) proving the conjectural statement from~\cite{arXiv1310.1362} 
on which that depends. Among other cases we consider, we include some for which
entries of $S$ may lie on the same row and the same column as other entries
in $S$. However, the complexity of the formulas increases rapidly. The largest case
covered by our results consists of the `thickened' diagonal:
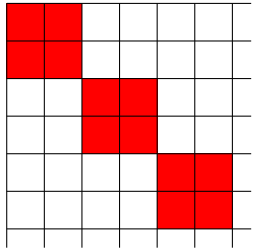
\begin{center}
\begin{tikzpicture}
\path [fill=red] (0,0) rectangle (1,-1);
\path [fill=red] (1,-1) rectangle (2,-2);
\path [fill=red] (2,-2) rectangle (3,-3);
\draw (0,0) --(0,-3.25);
\draw (.5,0) --(.5,-3.25);
\draw (1,0) --(1,-3.25);
\draw (1.5,0) --(1.5,-3.25);
\draw (2,0) --(2,-3.25);
\draw (2.5,0) --(2.5,-3.25);
\draw (3,0) --(3,-3.25);
\draw (0,0) --(3.25,0);
\draw (0,-.5) --(3.25,-.5);
\draw (0,-1) --(3.25,-1);
\draw (0,-1.5) --(3.25,-1.5);
\draw (0,-2) --(3.25,-2);
\draw (0,-2.5) --(3.25,-2.5);
\draw (0,-3) --(3.25,-3);
\end{tikzpicture}
\end{center}
See \S\ref{ss:22square} to get a sense of how involved the formulas become, even for such
comparatively mild sets $S$. In general, we can deal with sets $S$ consisting (up to 
permutations) of blocks with arbitrarily many entries in a row or a column, 
three adjacent entries, and $2\times 2$ squares.
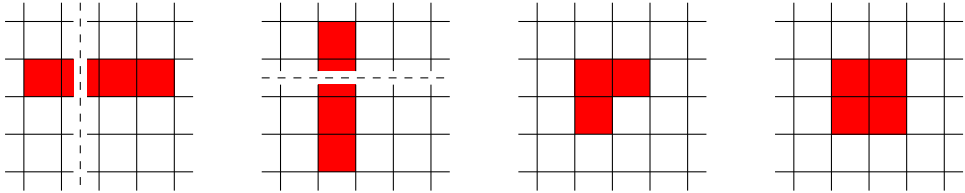
\begin{center}
\begin{tikzpicture}
\path [fill=red] (0,-.5) rectangle (0.66,-1);
\path [fill=red] (0.84,-.5) rectangle (2,-1);
\draw (0,.25) --(0,-2.25);
\draw (.5,.25) --(.5,-2.25);
\draw [dashed] (.75,.25) --(.75,-2.25);
\draw (1,.25) --(1,-2.25);
\draw (1.5,.25) --(1.5,-2.25);
\draw (2,.25) --(2,-2.25);
\draw (-.25,0) --(0.66,.0);
\draw (0.84,0) --(2.25,.0);
\draw (-.25,-.5) --(0.66,-.5);
\draw (0.84,-.5) --(2.25,-.5);
\draw (-.25,-1) --(0.66,-1);
\draw (0.84,-1) --(2.25,-1);
\draw (-.25,-1.5) --(0.66,-1.5);
\draw (0.84,-1.5) --(2.25,-1.5);
\draw (-.25,-2) --(0.66,-2);
\draw (0.84,-2) --(2.25,-2);
\end{tikzpicture}
\quad
\quad
\begin{tikzpicture}
\path [fill=red] (.5,0) rectangle (1,-0.66);
\path [fill=red] (.5,-0.84) rectangle (1,-2);
\draw (0,.25) --(0,-0.66);
\draw (0,-0.84) --(0,-2.25);
\draw (.5,.25) --(.5,-0.66);
\draw (.5,-0.84) --(.5,-2.25);
\draw (1,.25) --(1,-0.66);
\draw (1,-0.84) --(1,-2.25);
\draw (1.5,.25) --(1.5,-0.66);
\draw (1.5,-0.84) --(1.5,-2.25);
\draw (2,.25) --(2,-0.66);
\draw (2,-0.84) --(2,-2.25);
\draw (-.25,0) --(2.25,.0);
\draw (-.25,-.5) --(2.25,-.5);
\draw [dashed] (-.25,-.75) --(2.25,-.75);
\draw (-.25,-1) --(2.25,-1);
\draw (-.25,-1.5) --(2.25,-1.5);
\draw (-.25,-2) --(2.25,-2);
\end{tikzpicture}
\quad
\quad
\begin{tikzpicture}
\path [fill=red] (.5,-.5) rectangle (1.5,-1);
\path [fill=red] (.5,-1) rectangle (1,-1.5);
\draw (0,.25) --(0,-2.25);
\draw (.5,.25) --(.5,-2.25);
\draw (1,.25) --(1,-2.25);
\draw (1.5,.25) --(1.5,-2.25);
\draw (2,.25) --(2,-2.25);
\draw (-.25,0) --(2.25,.0);
\draw (-.25,-.5) --(2.25,-.5);
\draw (-.25,-1) --(2.25,-1);
\draw (-.25,-1.5) --(2.25,-1.5);
\draw (-.25,-2) --(2.25,-2);
\end{tikzpicture}
\quad
\quad
\begin{tikzpicture}
\path [fill=red] (.5,-.5) rectangle (1.5,-1.5);
\draw (0,.25) --(0,-2.25);
\draw (.5,.25) --(.5,-2.25);
\draw (1,.25) --(1,-2.25);
\draw (1.5,.25) --(1.5,-2.25);
\draw (2,.25) --(2,-2.25);
\draw (-.25,0) --(2.25,.0);
\draw (-.25,-.5) --(2.25,-.5);
\draw (-.25,-1) --(2.25,-1);
\draw (-.25,-1.5) --(2.25,-1.5);
\draw (-.25,-2) --(2.25,-2);
\end{tikzpicture}
\end{center}
The Segre class computations needed to deal with these cases are streamlined by the
use of a `twisted' version of Segre classes (Definition~\ref{def:twisSc}) closely related 
to a class introduced by Steven Kleiman and Anders Thorup (\cite{MR1393259}). 
The needed facts about this notion are summarized in Lemmas~\ref{lem:twifacts},
\ref{lem:scone}, and~\ref{lem:splay}. The tool that allows us to combine blocks of entries 
into larger sets $S$ stems from Lemma~\ref{lem:splay}, which is a variation on a formula 
for Segre classes of {\em splayed intersections\/} from~\cite{1406.1182}, and a formula 
for joins, Lemma~\ref{lem:scone}, which is an ingredient in the proof (given elsewhere)
of the rationality of a `zeta function' recording the behavior of Segre classes of ind-limits
of projective schemes (\cite{Segrat}).

For any set $S$ of entries in an $n\times n$ matrix, and any rank $r$, we define by
means of the twisted Segre class a class 
$\Sigma_{n,r,S}$ in the Chow group of the Grassmannian $G(n-r,n)$, which carries the 
information of the degree $d_{n,r,S}$. This is the {\em Grassmann class\/} of $S$ 
w.r.t.~$r$, Definition~\ref{def:GrassS}. We prove that
if $S$ consists of two blocks $S_1$, $S_2$ (in the sense specified 
in~\S\ref{ss:conessplayed}), then the corresponding Grassmann classes satisfy
a multiplicativity property:
\[
(1-\Sigma_{n,r,S}) = (1-\Sigma_{n,r,S'}) (1-\Sigma_{n,r,S''})
\]
(Theorem~\ref{thm:Sigblocks}). This multiplicative structure organizes the numbers
$d_{n,r,S}$ to some extent, but is only visible at the level of the corresponding
Grassmann classes. Analyzing the structure of these classes further appears to be
an interesting project.

We compute $\Sigma_{n,r,S}$ explicitly for the elementary blocks listed above.
Once $\Sigma_{n,r,S}$ is known for a set $S$, the degree $d_{n,r,S}$ may be computed 
as the intersection product $\int c(\cS^\vee)^n\cap \Sigma_{n,r,S}$ in the Grassmannian 
(Theorem~\ref{thm:mainred}). We express our results in terms of these intersection numbers;
the Macaulay2 package Schubert2 (\cite{S2}) can perform such computations at impressive 
speed. (Use of both Macaulay2 and Schubert2 was essential to us in this work.) Of course 
Macaulay2 (\cite{M2}) is capable of performing the degree computation directly; however, 
computation times become prohibitively long via such direct methods when the matrices 
reach even modest
size. Computations performed by implementing the results given here are much faster.
For example, let $S$ consist of six diagonal and three subdiagonal entries in
a $7\times 7$ matrix (Figure~\ref{fig:ladder}).
\begin{figure}
\begin{tikzpicture}
\path [fill=red] (0,0) rectangle (.5,-.5);
\path [fill=red] (0,-.5) rectangle (.5,-1);
\path [fill=red] (.5,-.5) rectangle (1,-1);
\path [fill=red] (1,-1) rectangle (1.5,-1.5);
\path [fill=red] (1,-1.5) rectangle (1.5,-2);
\path [fill=red] (1.5,-1.5) rectangle (2,-2);
\path [fill=red] (2,-2) rectangle (2.5,-2.5);
\path [fill=red] (2,-2.5) rectangle (2.5,-3);
\path [fill=red] (2.5,-2.5) rectangle (3,-3);
\draw (0,0) --(0,-3.5);
\draw (.5,0) --(.5,-3.5);
\draw (1,0) --(1,-3.5);
\draw (1.5,0) --(1.5,-3.5);
\draw (2,0) --(2,-3.5);
\draw (2.5,0) --(2.5,-3.5);
\draw (3,0) --(3,-3.5);
\draw (3.5,0) --(3.5,-3.5);
\draw (0,0) --(3.5,0);
\draw (0,-.5) --(3.5,-.5);
\draw (0,-1) --(3.5,-1);
\draw (0,-1.5) --(3.5,-1.5);
\draw (0,-2) --(3.5,-2);
\draw (0,-2.5) --(3.5,-2.5);
\draw (0,-3) --(3.5,-3);
\draw (0,-3.5) --(3.5,-3.5);
\end{tikzpicture}
\caption{}
\label{fig:ladder}
\end{figure}
Computing the degree of the projection of $\sigma_{7,2}$ (which happens to 
equal~$13395$; this is a dimension~$23$ subvariety of $\Pbb^{39}$)
took us about $20$ minutes by a `direct' method using Macaulay2 on on a quad-core 
($4 \times 3.4$GHz) computer, and about $.004$ seconds on the same equipment by 
implementing in Schubert2 the formulas proven here. (This example is covered by the 
formula given in Theorem~\ref{thm:manycorners}.) 
Thus, the results proven here allow for much more extensive experimentation with these 
numbers.

It would be interesting to perform the intersection products explicitly and obtain closed 
formulas for $d_{n,r,S}$ for all the cases considered here. We do this in some cases,
for example, when no two elements of $S$ are in the same row or column; this is the
case relevant to the conjecture from~\cite{arXiv1310.1362} mentioned above.

After compiling the results of many explicit computations, we noticed that many of the 
numbers $d_{n,r,S}$ appear in the existing literature in contexts far removed from the 
enumerative geometry of rank conditions; we owe this surprising (to us) observation 
to perusal of~\cite{oeis}.
For example, there is a persistent appearance of series reproducing the numbers of
{\em Kekul\'e structures\/} of many different types of {\em benzenoid 
hydrocarbons\/} (\cite{benzenoids1988kekule}). We ignore whether the combinatorial 
literature on these numbers links them directly to intersection products in Grassmannian. 
It would seem worthwhile to establish such a direct connection, and understand why these
numbers of interest in chemistry should arise as degrees of projections of rank 
loci. The numbers also appear 
in other contexts, for reasons that are equally mysterious to us. See~\S\ref{ss:benzene},
\S\ref{ss:benzene2}, Remark~\ref{rem:onecor}, Example~\ref{ex:squareexs}.
\medskip

It would be straightforward to extend the results presented here to rank loci in the
space $\Pbb^{mn-1}$ of $m\times n$ matrices.
\medskip

{\em Acknowledgments.} The author thanks J.~M.~Landsberg for the invitation to visit 
Texas A\&M and for discussions leading to the question addressed in this paper.
The author's research is partially supported by a Simons collaboration grant.
\newpage


\section{Preliminaries}\label{prelims}

\subsection{}\label{ss:intropre}
We work over an algebraically closed field $\kappa$.
We identify the space of $n\times n$ matrices with entries in $\kappa$, up to scalars, with 
$\Pbb^{n^2-1}$; the entries of the matrices are homogeneous coordinates in this
projective space. For a set $S$ of $s\ge 1$ entries, $L_S$ will denote the linear subspace
of $\Pbb^{n^2-1}$ spanned by $S$.

We let $\sigma_{n,r}\subseteq \Pbb^{n^2-1}$ denote the set of $n\times n$ matrices of 
rank~$\le r$. It is often notationally more convenient to work with the {\em co\/}rank 
$k=n-r$; we will let $\tau_{n,k}=\sigma_{n,n-k}$. Thus, $\tau_{n,k}$ consists of the set of 
$n\times n$ matrices $A$ such that $\dim \ker A\ge k$. The algebraic set $\tau_{n,k}$
is defined (with its reduced structure) by the ideal generated by the $n-k+1$ minors
of $A\in \Pbb^{n^2-1}$. A simple dimension count shows that $\dim \tau_{n,k} = n^2
-k^2-1$. Also,
\begin{equation}\label{eq:degtau}
\deg \tau_{n,k} = \prod_{i=0}^{k-1} \frac{\binom{n+i}k}{\binom{k+i}k}
\end{equation}
(Example~14.4.14 in~\cite{85k:14004}).

The subspace $L_S$ defines a projection $\pi_S: \Pbb^{n^2-1} \to \Pbb^{n^2-s-1}$. 
The closure of $\pi_S(\tau_{n,k})$ is an irreducible subvariety $\tau_{n,k,S}$ of 
$\Pbb^{n^2-s-1}$. Briefly stated, our goal is the computation of the degree of this 
subvariety. We begin by formulating this task more precisely.

\subsection{}\label{ss:dnrS}
Let $T_{n,k,S}$ be any variety mapping properly birationally to $\tau_{n,k}$ and 
resolving the indeterminacies of the restriction of $\pi_S$ to $\tau_{n,k}$:
\[
\xymatrix{
\Pbb^{n^2-1} & \tau_{n,k} \ar@{_(->}[l] \ar@{-->}[rr]^{{\pi_S}|_{\tau_{n,k}}} & & 
\tau_{n,k,S} \ar@{^(->}[r] & \Pbb^{n^2-s-1} \\
& & T_{n,k,S} \ar[ul] \ar@{->>}[ur]_{\tilde\pi_S}
}
\]
For example, we can take $T_{n,k,S}$ to be the blow-up of $\tau_{n,k}$ along 
the intersection $L_S\cap \tau_{n,k}$. We let $\tilde\pi_S$ be the induced regular 
map $T_{n,k,S} \to \tau_{n,k,S}$.
Note that the general fiber of $\tilde\pi_S$ is finite if and only if so is the fiber of 
$\tau_{n,k,S}$ (away from $L_S$), and $\deg\tilde\pi_S$ equals the number of 
points in the general fiber of ${\pi_S}|_{\tau_{n,k}}$ and not on $L_S$ in that case.

Let $h$ be the hyperplane class in $\Pbb^{n^2-s-1}$. 
\begin{defin}\label{def:dnrS}
We let
\[
d_{n,r,S}:=\int_{T_{n,k,S}} ({\tilde\pi_S}^*h)^{n^2-k^2-1}
\]
where $k=n-r$.
\end{defin}

\begin{lemma}
The number $d_{n,r,S}$ equals $f\cdot \deg \tau_{n,k,S}$, where $f$ is the degree of 
$\tilde\pi_S$ (and $f=0$ if the general fiber of $\tilde\pi_S$ has positive dimension).
In particular, $d_{n,r,S}$ is independent of the chosen variety $T_{n,k,S}$.
\end{lemma}

\begin{proof}
The first point is immediate from the definition, and the second point follows since
$f$ (if $\ne 0$) equals the number of points in a general fiber of $\pi_S|_{\tau_{n,k}}$
in the open set $\Pbb^{n^2-1}\smallsetminus L_S$. Alternately, the second point follows
from the projection formula, since two resolutions $T_{n,k,S}$ are dominated by
a third one.
\end{proof}

\begin{remark}\label{rem:perm}
The number $d_{n,r,S}$ is evidently invariant under permutations of the rows or columns 
of the matrix.
\qede\end{remark}

Our main goals in this note are to present a method for the computation of $d_{n,r,S}$
and to apply this method to several concrete cases. In principle the method can be applied
to more general situations than those that will be considered explicitly here, but the
complexity of the needed calculations increases rapidly.

\begin{remark}
When $\tilde\pi_S$ is birational onto its image $\tau_{n,k,s}$, then $d_{n,r,S}$ simply equals
the degree of $\tau_{n,k,S}$ as a subvariety of $\Pbb^{n^2-s-1}$. This degree also equals the
degree of the cone over $\tau_{n,k,S}$ with vertex $L_S$, that is, the {\em join\/} of $L_S$
and $\tau_{n,k}$. The language of joins is preferred in~\cite{arXiv1310.1362}, and is
essentially equivalent to the one used here.
\qede\end{remark}

\begin{remark}\label{rem:deq0}
There are of course situations in which $d_{n,r,S}$ does {\em not\/} equal the degree of
$\tau_{n,k,S}$. For example, if the general fiber of ${\pi_S}|_{\tau_{n,k}}$ is positive 
dimensional, then $\dim \tau_{n,k,S}<n^2-k^2-1$, and $d_{n,r,S}=0$.
This is necessarily the case if $s>k^2$, but can occur in more interesting situations.
For example, the general fiber of ${\pi_S}|_{\tau_{n,k}}$ is positive dimensional for~$S$ 
consisting of $s> k=n-r$ entries in one row:
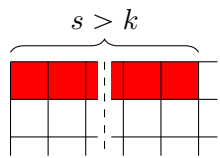
\begin{center}
\begin{tikzpicture}
\path [fill=red] (0,-0) rectangle (1.16,-.5);
\path [fill=red] (1.34,-0) rectangle (2.5,-.5);
\draw (0,0) --(0,-1.25);
\draw (.5,0) --(.5,-1.25);
\draw (1,0) --(1,-1.25);
\draw (1.5,0) --(1.5,-1.25);
\draw (2,0) --(2,-1.25);
\draw (2.5,0) --(2.5,-1.25);
\draw [dashed] (1.25,0) --(1.25,-1.25);
\draw (0,0) --(1.16,0);
\draw (0,-.5) --(1.16,-.5);
\draw (0,-1) --(1.16,-1);
\draw (1.34,-0) --(2.75,-0);
\draw (1.34,-.5) --(2.75,-.5);
\draw (1.34,-1) --(2.75,-1);
	\draw [decorate,decoration={brace,amplitude=5pt},
			xshift=0pt, yshift=0pt]
			(0,.125) -- (2.5,0.125)
			node[black,midway,xshift=0pt,yshift=12.5pt]
				{$s>k$};
\end{tikzpicture}
\end{center}
This could be verified by geometric means, but it is an immediate consequence of the formulas
obtained in this note, since $d_{n,r,S}$ will turn out to vanish in this case 
(Example~\ref{ex:exarow}).
Thus, $d_{n,r,S}$ (as defined in Definition~\ref{def:dnrS}) carries interesting geometric 
information, and we choose to focus on it rather than on the degree of $\tau_{n,k,S}$ proper. 
\qede\end{remark}

\begin{remark}\label{rem:dto1}
It may also occur that the general fiber of $\tilde\pi_S$ has dimension zero, but degree larger
than $1$.
For example, our formulas will give that $d_{7,4,S}=2$ for the ladder configuration shown 
in Figure~\ref{fig:ladder}. In this case $s=9=k^2$, so the projection of $\sigma_{7,4}$ must 
be generically $2$-to-$1$ onto $\Pbb^{n^2-s-1}=\Pbb^{39}$. Again this information is 
immediately accessible from $d_{7,4,S}=2$ (and not from $\deg \tau_{7,4,S}=1$).
For an instance of the same phenomenon in which the projection is not dominant, 
see Example~\ref{ex:mixed}.

It would be interesting to have a criterion indicating when $\deg\tilde\pi_S=1$. According
to Lemma~5.4.1 in~\cite{arXiv1310.1362}, this is the case if e.g.~$S$ is a subset of
the diagonal and $s<(n-r)^2$. 

Even for diagonal $S$, the degree of $\deg\tilde\pi_S$ is not $1$ in general
if $s=(n-r)^2$: for example, it is $2$ for $r=2$, $n=s=2^2=4$ and $42$ for
$r=6$, $n=s=3^2=9$ (since $d_{4,2,S}=2$, $d_{9,6,S}=42$ if $S$ is the full
diagonal, see Example~\ref{ex:maxdia}).
In fact, numerical experimentation indicates that this degree is {\em independent of $n$\/} 
for $S$ consisting of $(n-r)^2$ diagonal elements, a fact for which we do not have a
proof.
\qede\end{remark}

\subsection{}\label{sec:excint}
We will treat the problem of computing $d_{n,r,S}$ as a problem of {\em excess intersection.\/} 
Let $N=\dim\tau_{n,k}=n^2-k^2-1$. The intersection of $N$ general hyperplanes $h_1,\dots, h_N$ 
in $\Pbb^{n^2-s-1}$ meets $\tau_{n,k,S}$ at reduced points by Bertini, and it follows 
that $d_{n,r,S}$ equals the number of points of intersection of $\tilde\pi_S^{-1}(h_1),\dots, 
\tilde\pi_S^{-1}(h_N)$. The hyperplanes $h_i$ correspond to general hyperplanes $H_i$ of 
$\Pbb^{n^2-1}$ {\em containing $L_S$,\/} and it follows that the intersection
\[
H_1\cap \cdots \cap H_N\cap \tau_{n,k}
\]
consists of $d_{n,r,S}$ isolated reduced points and of a scheme supported on $L_S\cap
\tau_{n,k}$. The intersection number
\begin{equation}\label{eq:basicint}
\int H_1\cdot \cdots \cdot H_N\cdot \tau_{n,k}
\end{equation}
equals $\deg \tau_{n,k}$ (given in~\eqref{eq:degtau}); therefore
\[
d_{n,r,S} = \prod_{i=0}^{k-1} \frac{\binom{n+i}k}{\binom{k+i}k} - \text{contribution
of $(L_S\cap \tau_{n,k})$ to the intersection product~\eqref{eq:basicint}.}
\]
The challenge consists therefore of evaluating the contribution of the ({\em a priori\/}
positive dimensional) locus $L_S\cap \tau_{n,k}$ to the intersection product of $\tau_{n,k}$
and $N$ general hyperplanes containing $L_S$. By Fulton-MacPherson intersection theory,
this contribution may be expressed in terms of a Segre class.

\begin{lemma}\label{lem:FMred}
Let $H$ be the class of a hyperplane in $\Pbb^{n^2-1}$. Then
\begin{equation}\label{eq:firstred}
d_{n,r,S} = \prod_{i=0}^{k-1} \frac{\binom{n+i}k}{\binom{k+i}k} 
-\int (1+H)^{n^2-k^2-1}\cap s(L_S\cap \tau_{n,k},\tau_{n,k})\quad.
\end{equation}
\end{lemma}

Here $s(L_S\cap \tau_{n,k},\tau_{n,k})$ is the {\em Segre class\/} in the sense of
\cite{85k:14004}, Chapter 4. In~\eqref{eq:firstred}, $L_S\cap \tau_{n,k}$ is taken with its 
natural (not necessarily reduced!) scheme structure, i.e., the one given by the sum of 
the ideals of $L_S$ and $\tau_{n,k}$.

\begin{proof}
Notice that $L_S\cap \tau_{n,k}$ is the base scheme of the linear system cut out on
$\tau_{n,k}$ by the hyperplanes containing $L_S$. Then apply Proposition~4.4 
in~\cite{85k:14004}.
\end{proof}

The problem shifts then to the computation of the Segre class $s(L_S\cap \tau_{n,k},
\tau_{n,k})$. 

\subsection{}
In order to obtain more manageable formulas, we will use Lemma~\ref{lem:FMred} in
a slightly modified version, `twisting' the needed Segre class by means of an operation
introduced in~\cite{MR96d:14004}.

\begin{defin}\label{def:twisSc}
Let $Z$ be a subscheme of a variety $V$, and let $\cL$ be a line bundle on $Z$.
We let
\begin{equation}\label{eq:twisted}
s(Z,V)^\cL := c(\cL)^{-1}\cap (s(Z,V)\otimes_V \cL)\quad.
\end{equation}
\end{defin}

The notation used in the right-hand side of~\eqref{eq:twisted} is defined 
in~\S2 of~\cite{MR96d:14004}, to which
we address the reader for further details (which will be immaterial here). In practice,
$s(Z,V)^{\cL}$ is obtained by capping by $c(\cL)^{-(c+1)}$ the piece of $s(Z,V)$ of
codimension~$c$ {\em in $V$,\/} for all $c$.

\begin{remark}
The class introduced in Definition~\ref{def:twisSc} is closely related to the {\em twisted
Segre class\/} studied in~\cite{MR1393259}.
\qede\end{remark}

We collect here a few facts concerning this notion, that we will need in the next sections.

\begin{lemma}\label{lem:twifacts}
\begin{enumerate}
\item\label{pt1} If $Z$ is regularly embedded in $V$, then
\[
s(Z,V)^{\cL}=c(\cL)^{-1} c(N_ZV\otimes \cL)^{-1}\cap [Z]\quad.
\]
\item\label{pt2} If $\pi: \Til V\to V$ is a proper birational morphism, and $\Til \cL$ is the 
pull-back to $\pi^{-1}(Z)$ of the line bundle $\cL$, then
\[
s(Z,V)^\cL = \pi_* \left(s(\pi^{-1}(Z), \Til V)^{\Til \cL}\right)\quad.
\]
\item\label{pt3} If $V$ is nonsingular, then the class
\[
c(TV|_Z\otimes \cL)\cap s(Z,V)^\cL
\]
only depends on $Z$ and $\cL$.
\item\label{pt4} With notation as in~\S\ref{sec:excint},
\[
d_{n,r,S} = \prod_{i=0}^{k-1} \frac{\binom{n+i}k}{\binom{k+i}k} 
-\int s(L_S\cap \tau_{n,k},\tau_{n,k})^{\cO(-H)}\quad.
\]
\end{enumerate}
\end{lemma}

Full proofs of these facts will be given elsewhere.
\eqref{pt1}, \eqref{pt2} are formal consequence of the definition, the standard formula for 
Segre classes of regularly embedded schemes (see Chapter~4 of~\cite{85k:14004}), and 
of the birational invariance of Segre classes (Proposition~4.2 (a) in~\cite{85k:14004}). 
\eqref{pt3} follows from Example~4.2.6 (a) in~\cite{85k:14004}. (4) is again a formal 
consequence of the definition and of Lemma~\ref{lem:FMred}.

By \eqref{pt4}, the problem of computing $d_{n,r,S}$ is reduced to the computation
of the class $s(L_S\cap \tau_{n,k},\tau_{n,k})^{\cO(-H)}$. We will use 
\eqref{pt2} to relate this computation to one in an ambient nonsingular variety and 
\eqref{pt3} to change this ambient nonsingular variety to a more convenient one. 
In particularly simple situations, \eqref{pt1} will suffice to compute the needed class. 

We will need somewhat more sophisticated facts concerning $s(Z,V)^\cL$ 
in~\S\ref{ss:conessplayed} (Lemma~\ref{lem:scone}, Lemma~\ref{lem:splay}).

\subsection{}\label{ss:resolution}
By Lemma~\ref{lem:twifacts}~\eqref{pt4}, our objective consists of the computation of
$s(L_S\cap \tau_{n,k},\tau_{n,k})^{\cO(-H)}$; by~\eqref{pt2} in the same lemma, we
may replace $\tau_{n,k}$ with any birational model $\htau_{n,k}$. We choose for
$\htau_{n,k}$ the standard resolution of $\tau_{n,k}$ as a projective bundle over 
a Grassmannian.

Briefly, $\htau_{n,k}$ consists of pairs $(K,\varphi)\in G(k,n)\times \Pbb^{n^2-1}$
where $K\in G(k,n)$ is a $k$-dimensional subspace of $\kappa^n$, and $\varphi$ is
a matrix such that $K\subseteq \ker \varphi$ (and hence $\rk\varphi\le r=n-k$). 
More intrinsically, consider the tautological sequence
\[
\xymatrix{
0 \ar[r] & \cS \ar[r] & \cO^n \ar[r] & \cQ \ar[r] & 0
}
\]
over the Grassmannian $G(k,n)$. 
The space $\Hom(\kappa^n,\kappa^n)$ of matrices defines the trivial bundle $\Hom(\cO^n,
\cO^n)$ over $G(k,n)$. For $K\in G(k,n)$, the homomorphisms whose kernel 
contains $K$ correspond to those $\varphi: \kappa^n \to \kappa^n$ that are induced from a 
homomorphism on the quotient $\kappa^n/K$. Globalizing, the $\Pbb^{n(n-k)-1}$-bundle 
$\htau_{n,k}$ on $G(k,n)$ whose fiber over $K$ parametrizes matrices $\varphi$ for
which $K\subseteq \ker\varphi$ is the projectivization of
\[
\Hom(\cQ,\cO^n)\cong (\cQd)^n\quad.
\]
The embedding $\cQd \subseteq \cO^n$ induces an inclusion
\[
\htau_{n,k} \subseteq \Pbb\cO^{n^2} = G(k,n) \times \Pbb^{n^2-1}\quad;
\]
the projection to $\Pbb^{n^2-1}$ maps $\htau_{n,k}$ onto $\tau_{n,k}$:
\[
\xymatrix@C=10pt{
& \htau_{n,k} \ar[ld]_\rho \ar[rd]^\nu \ar@{^(->}[r] & G(k,n)\times \Pbb^{n^2-1} \ar[dr] \\
G(k,n) & & \tau_{n,k} \ar@{^(->}[r] & \Pbb^{n^2-1}
}
\]
Both $\tau_{n,k}$ and $\htau_{n,k}$ are projective, and every $\varphi\in\tau_{n,k}$ 
with rank {\em exactly\/} $r=n-k$ has a well-defined kernel $K\in G(k,n)$; thus $\nu:
\htau_{n,k} \to \tau_{n,k}$ is a proper birational morphism, resolving the singularities 
of $\tau_{n,k}=\sigma_{n,r}$.
(In particular $\dim \tau_{n,k}=\dim G(k,n)+n(n-k)-1=k(n-k)+n(n-k)-1=n^2-k^2-1$ 
as stated in~\S\ref{ss:intropre}.)

As in the diagram, we let $\rho: \htau_{n,k} = \Pbb((\cQd)^n) \to G(k,n)$ be the
structure map. Note that the universal bundle $\cO(1)$ on $\Pbb((\cQd)^n)$ is 
the restriction of the hyperplane bundle $\cO(H)$ from the second factor in
$G(k,n)\times \Pbb^{n^2-1}$.

The following observation will be useful in concrete calculations:
\begin{lemma}\label{lem:degcom}
\[
\deg\sigma_{n,r}=\deg \tau_{n,k} = \prod_{i=0}^{k-1} \frac{\binom{n+i}k}{\binom{k+i}k} = \int_{G(k,n)}
c(\cSd)^n\quad.
\]
\end{lemma}

\begin{proof}
The numerical value for the degree was already recalled in~\S\ref{ss:intropre}. To see
that this is the degree of $c(\cSd)^n$, we may argue as follows. By definition,
\[
\deg\tau_{n,k}=\int H^{\dim \tau_{n,k}} \cap [\tau_{n,k}]
=\int \sum_i H^i \cap [\tau_{n,k}]\quad;
\]
by the projection formula (applied to $\nu$), this equals
\[
\int \sum_i c_1(\cO(1))^i \cap [\Pbb((\cQd)^n)]\quad,
\]
since as observed above $H$ pulls back to $c_1(\cO(1))$ on $\htau_{n,k}
=\Pbb((\cQd)^n)$. Since degrees are preserved by proper push-forwards,
the degree in turn equals
\[
\int \rho_*\left(\sum_i c_1(\cO(1))^i \cap [\Pbb((\cQd)^n)]\right)
=\int s((\cQd)^n)\cap [G(k,n)]
\]
by definition of Segre class of a vector bundle. By the Whitney formula applied to the 
(dual of the) tautological sequence we have
$s(\cQd)=c(\cSd)$, and this concludes the proof.
\end{proof}

\subsection{}
Applying Lemma~\ref{lem:twifacts}~\eqref{pt2} and~\eqref{pt4}, we have
\begin{equation}\label{eq:predeg}
d_{n,r,S} = \prod_{i=0}^{k-1} \frac{\binom{n+i}k}{\binom{k+i}k} 
-\int s(\nu^{-1}(L_S\cap \tau_{n,k}), \htau_{n,k})^{\cO(-H)}
\end{equation}
where of course $H$ denotes the pull-back of the hyperplane class to $\htau_{n,k}$.

Now let $\hL_S=G(k,n)\times L_S\cong G(k,n)\times \Pbb^{s-1}$, and note that
$\nu^{-1}(L_S\cap \tau_{n,k}) = \hL_S\cap \htau_{n,k}$.
\[
\xymatrix@C=20pt{
& \htau_{n,k} \ar[ld]_\rho \ar[rd]^\nu \ar@{^(->}[r] & G(k,n)\times \Pbb^{n^2-1} \ar[dr] 
& \hL_S = G(k,n)\times L_S \ar@{_(->}[l] \ar[dr] \\
G(k,n) & & \tau_{n,k} \ar@{^(->}[r] & \Pbb^{n^2-1} & L_S \ar@{_(->}[l]
}
\]
Denote by $\pi_1: \hL_S \to G(k,n)$ the projection. We note that the push-forward
$\pi_{1*}$ acts by reading off the coefficient of $H^{s-1}$.
The information needed in order
to compute the degrees $d_{n,r,S}$ is captured by the following class in the Chow
group of the Grassmannian $G(k,n)$.

\begin{defin}\label{def:GrassS}
The {\em Grassmann class\/} $\Sigma_{n,r,S}$ associated with $S$ (w.r.t.~$r$) is the 
push-forward
\[
\Sigma_{n,r,S}:=\pi_{1*}\left( s(\hL_S\cap \htau_{n,k},\hL_S)^{\cO(-H)} \right)
\]
of the `twisted' Segre class $s(\hL_S\cap \htau_{n,k},\hL_S)^{\cO(-H)}$, where $k=n-r$.
\end{defin}

Note that we are taking the Segre class {\em in $\hL_S$\/} in this definition (rather than
{\em in $\htau_{n,k}$}).
The following result will be our main tool in the rest of the paper.

\begin{theorem}\label{thm:mainred}
With notation as above,
\begin{equation}\label{eq:mainred}
d_{n,r,S} = \int_{G(n-r,n)} c(\cSd)^n\cap \left(1-\Sigma_{n,r,S}\right)\quad.
\end{equation}
\end{theorem}

\begin{proof}
By Lemma~\ref{lem:twifacts}~\eqref{pt3},
\[
s(\hL_S\cap \htau_{n,k}, \htau_{n,k})^{\cO(-H)}
=\frac{c(T\hL_S\otimes \cO(-H))}{c(T\htau_{n,k}\otimes \cO(-H))}
\cap s(\hL_S\cap \htau_{n,k}, \hL_S)^{\cO(-H)}
\]
The Chern classes of both bundles appearing on the right are evaluated by standard 
Euler sequences. Since
$\hL_S$ is a trivial $\Pbb^{s-1}$-bundle over $G(k,n)$, we have the exact sequence 
\[
\xymatrix{
0 \ar[r] & \cO \ar[r] & \cO(H)^s \ar[r] & T\hL_S \ar[r] & TG(k,n) \ar[r] & 0
}
\]
(omitting the evident pull-back). Tensoring by $\cO(-H)$ we get the exact sequence
\[
\xymatrix{
0 \ar[r] & \cO(-H) \ar[r] & \cO^s \ar[r] & T\hL_S\otimes \cO(-H) \ar[r] & TG(k,n)\otimes
\cO(-H) \ar[r] & 0
}
\]
from which
\[
c(T\hL_S\otimes \cO(-H))=c(\cO(H))^{-1} c(TG(k,n)\otimes \cO(-H))\quad.
\]
Tensoring the Euler sequence for $\Pbb((\cQd)^n)$ by $\cO(-H)$ (and omitting
pull-backs and restrictions by a common abuse of language) gives likewise
\[
\xymatrix{
0 \ar[r] & \cO(-H) \ar[r] & (\cQd)^n \ar[r] & T\htau_{n,k}\otimes \cO(-H) \ar[r] & 
TG(k,n)\otimes \cO(-H) \ar[r] & 0
}
\]
from which
\[
c(T\htau_{n,k}\otimes \cO(-H))=c(\cO(H))^{-1} c(\cQd)^n c(TG(k,n)\otimes \cO(-H))\quad.
\]
Therefore
\[
\frac{c(T\hL_S\otimes \cO(-H))}{c(T\htau_{n,k}\otimes \cO(-H))}
=\frac{1}{c(\cQd)^n} = c(\cSd)^n\quad,
\]
and \eqref{eq:predeg} gives
\[
d_{n,r,S} = \prod_{i=0}^{k-1} \frac{\binom{n+i}k}{\binom{k+i}k} 
-\int_{\hL_S} c(\cSd)^n \cap s(\hL_S\cap \htau_{n,k}, \hL_S)^{\cO(-H)}
\]
Since $\int$ is not affected by proper push-forwards,
\[
d_{n,r,S} = \prod_{i=0}^{k-1} \frac{\binom{n+i}k}{\binom{k+i}k} 
-\int_{G(k,n)} c(\cSd)^n \cap \pi_{1*}\left(s(\hL_S\cap \htau_{n,k}, \hL_S)^{\cO(-H)}\right)
\]
and by Lemma~\ref{lem:degcom} we then have
\[
d_{n,r,S} = \int_{G(k,n)} c(\cSd)^n \left(1 - \pi_{1*}\left( s(\hL_S\cap \htau_{n,k}, 
\hL_S)^{\cO(-H)}\right)\right)\quad,
\]
which is the statement.
\end{proof}

In view of expression~\eqref{eq:mainred} obtained in 
Theorem~\ref{thm:mainred}, the key ingredient in the calculation of $d_{n,r,S}$ in
our approach is
the computation of the Grassmann class $\Sigma_{n,r,S}$ determined by the choice
of $n,r,S$. In the rest of the paper we will compute this class explicitly in several 
template cases, and
prove a result relating $\Sigma_{n,r,S}$ to the Grassmann classes of smaller sets,
provided these form `blocks' for $S$. Once $\Sigma_{n,r,S}$ is obtained, tools
such as Schubert2 can perform the degree computation required by~\eqref{eq:mainred}
and obtain numerical values for $d_{n,r,S}$ in many cases of reasonable size. In some cases
(particularly when $S$ is `diagonal') we will be able to perform the degree computation
explicitly in general.


\section{Entries from a single row or column}\label{singlerow}

\subsection{}\label{ss:singlerowlocus}
Let $S$ consist of $s=\ell$ entries in a single row of the matrix. By Remark~\ref{rem:perm},
we could choose these entries to be adjacent and left-adjusted in the top row of the matrix. 
\begin{center}
\begin{tikzpicture}
\path [fill=red] (0,-0) rectangle (1.16,-.5);
\path [fill=red] (1.34,-0) rectangle (2.5,-.5);
\draw (0,0) --(0,-1.25);
\draw (.5,0) --(.5,-1.25);
\draw (1,0) --(1,-1.25);
\draw (1.5,0) --(1.5,-1.25);
\draw (2,0) --(2,-1.25);
\draw (2.5,0) --(2.5,-1.25);
\draw (3,0) --(3,-1.25);
\draw (3.5,0) --(3.5,-1.25);
\draw [dashed] (1.25,0) --(1.25,-1.25);
\draw (0,0) --(1.16,0);
\draw (0,-.5) --(1.16,-.5);
\draw (0,-1) --(1.16,-1);
\draw (1.34,-0) --(3.75,-0);
\draw (1.34,-.5) --(3.75,-.5);
\draw (1.34,-1) --(3.75,-1);
	\draw [decorate,decoration={brace,amplitude=5pt},
			xshift=0pt, yshift=0pt]
			(0,.125) -- (2.5,0.125)
			node[black,midway,xshift=0pt,yshift=12.5pt]
				{$\ell$};
\end{tikzpicture}
\end{center}
(We denote $s$ by $\ell$ in this section, for compatibility with notation used in later sections.)
Recall that $\hL_S$ denotes 
$G(k,n)\times \Pbb^{\ell-1}$; we need to compute $s(\hL_S\cap \htau_{n,k},\hL_S)^{\cO(-H)}$, 
where $\htau_{n,k}$ is the standard resolution of $\tau_{n,k}=\sigma_{n,r}$, the variety of 
matrices with rank $\le r=n-k$ (see~\S\ref{ss:resolution}). The kernel of a matrix in the span 
of $S$,
\[
\varphi=\begin{pmatrix}
a_{11} & \cdots & a_{1\ell} & 0 & \cdots & 0 \\
0 & \cdots & 0 & 0 & \cdots & 0 \\
\vdots & \ddots & \vdots & \vdots & \ddots & \vdots \\
0 & \cdots & 0 & 0 & \cdots & 0
\end{pmatrix}\quad,
\]
is the hyperplane $T_\varphi$ with equation $a_{11}x_1+\cdots +a_{1\ell}x_\ell=0$. By definition 
of $\htau_{k,n}$ (see~\S\ref{ss:resolution}), the scheme $\hL_S\cap \htau_{k,n}$ is supported 
on the Grassmann bundle over $\Pbb^{\ell-1}$ whose fiber over $\varphi$ is the set of
$k$-planes $K$ contained in $T_\varphi$.

\begin{lemma}\label{lem:onerowdes}
The intersection $\hL_S\cap \htau_{k,n}$ equals this Grassmann bundle scheme-theoretically;
in particular, it is a nonsingular subvariety of $\hL_S$. It is the zero-scheme of a regular
section of $\cSd\otimes \cO(H)$ over $\hL_S$.
\end{lemma}

\begin{proof}
The first assertion may be checked by a coordinate computation. We work in a neighborhood
of the (arbitrarily chosen) point $(K_0,\varphi)\in G(k,n)\times \Pbb^{n^2-1}$ with $K_0$ spanned
by the last $k$ coordinates and $\varphi$ given by $a_{11}=1, a_{ij}=0$ for $(i,j)\ne (1,1)$.
Local coordinates at this point are
\[
\left(
C=\begin{pmatrix}
c_{11} & \dots & c_{1,n-k} & 1 & \dots & 0\\
\vdots & \ddots & \vdots & \vdots & \ddots & \vdots \\
c_{k1} & \dots & c_{k,n-k} & 0 & \dots & 1
\end{pmatrix}
, 
A=\begin{pmatrix}
1 & a_{12} & \dots & a_{1n} \\
a_{21} & a_{22} & \dots & a_{2n} \\
\vdots &  \vdots & \ddots & \vdots \\
a_{n1} & a_{n2} & \dots & a_{nn}
\end{pmatrix}
\right)
\]
and equations for $\htau_{k,n}$ are $A\cdot C^t=0$. Coordinates for $\hL_S$ are
as above, with all $a_{1\ell+1},\dots, a_{1n}$ and $a_{21},\dots, a_{nn}$ set to $0$. 
Therefore, local generators for the ideal of $\hL_S\cap \htau_{k,n}$ in $\hL_S$ are
\begin{equation}\label{eq:equarow}
\left\{\aligned
c_{11}  + c_{12} a_{12} + \dots + c_{1 \ell} a_{1\ell} &= 0 \\
\cdots \quad\quad\quad\quad &\\
c_{k1} + c_{k2} a_{12} + \dots + c_{k\ell } a_{1\ell} &= 0
\endaligned\right.
\end{equation}
(with obvious adaptation if $\ell\ge n-k$).
This shows that $\hL_S\cap \htau_{k,n}$ is nonsingular, and in fact a local complete
intersection of codimension~$k$, and it follows that it agrees with its support, described 
set-theoretically as a Grassmann bundle in the paragraph preceding the statement.

For the second part of the statement, we work more intrinsically. The inclusion of the
space $L_S$ in the span of the top rows determines an inclusion $\cO^\ell \subseteq
\cO^n$ of bundles over $\Pbb^{\ell-1}$. A point $\varphi\in \Pbb^{\ell-1}$ corresponds
to vectors $\hat\varphi$ in the fiber of $\cO(-1)\subseteq \cO^\ell \subseteq \cO^n$;
dualizing, we get a morphism $\cO^n \to \cO(1) = \cO(H)$. Let $T$ be the kernel of this
morphism, so with notation as above, the hyperplane $T_\varphi$ is the fiber of $T$
over $\varphi$, and $\hL_S\cap \htau_{k,n}$ is the Grassmann bundle $G(k,T)$ over
$\Pbb^{\ell-1}$. By construction, $G(k,T)$ is the zero-scheme of the section of 
$\cSd \otimes \cO(H)=\Hom(\cS,\cO(H))$ determined by composition:
\[
\cS \longrightarrow \cO^n \longrightarrow \cO(H)\quad.
\]
This section is regular, as the above local coordinate computation shows.
\end{proof}

\begin{corol}\label{cor:onerowclass}
Let $S$ consist of $\ell$ entries in a row. We have
\[
s(\hL_S\cap \htau_{k,n}, \hL_S)^{\cO(-H)} 
= c(\cO(-H))^{-1} c(\cSd)^{-1}\cap [\hL_S\cap \htau_{k,n}]
\quad.
\]
As a class in $\hL_S=G(k,n)\times \Pbb^{\ell-1}$, 
\begin{equation}\label{eq:onerow}
s(\hL_S\cap \htau_{k,n}, \hL_S)^{\cO(-H)} 
= c(\cO(-H))^{-1} c(\cSd)^{-1} c_{\text{top}}(\cSd\otimes \cO(H))
\quad.
\end{equation}
\end{corol}

\begin{proof}
By Lemma~\ref{lem:onerowdes}, $\hL_S\cap \htau_{k,n}$ is regularly embedded in
$\hL_S$, with class $c_{\text{top}}(\cSd\otimes \cO(H))$ and normal bundle 
$\cSd\otimes \cO(H)$. The statement then follows from Lemma~\ref{lem:twifacts} 
\eqref{pt1}.
\end{proof}

\begin{corol}\label{cor:SigmaSrow}
Let $S$ consist of $\ell$ entries in a row. Then
\begin{equation}\label{eq:SigmaSrow}
\Sigma_{n,r,S} 
=c(\cQd)  \left( c_{n-r-\ell+1}(\cSd) + \cdots + c_{n-r}(\cSd)\right)\quad.
\end{equation}
\end{corol}

\begin{proof}
By definition, $\Sigma_{n,r,S}=\pi_{1*}s(\hL_S\cap \htau_{k,n}, \hL_S)^{\cO(-H)}$.
Recall that $\pi_1$ denotes the projection $\hL_S\to G(k,n)$, and it acts by collecting
the coefficient of $H^{\ell-1}$. 
We have
\begin{equation}\label{eq:tenpro}
c_{\text{top}}(\cSd\otimes \cO(H)) = \sum_{i=0}^k H^i\cdot c_{k-i}(\cSd)
\end{equation}
(p.~55 in~\cite{85k:14004}), and it follows that the coefficient of $H^{\ell-1}$ in
\[
c(\cO(-H))^{-1} c_{\text{top}}(\cSd\otimes \cO(H)) = \sum_{j\ge 0} H^j \cdot
\sum_{i=0}^k H^i\cdot c_{k-i}(\cSd)
\]
equals $c_k(\cSd) + \cdots + c_{k-\ell+1}(\cSd)$. Therefore
\begin{multline*}\label{eq:Sigmarow}
\Sigma_{n,r,S} = \pi_{1*} \left( c(\cSd)^{-1} c(\cO(-H))^{-1}\cap 
c_{\text{top}}(\cSd\otimes \cO(H)) \right) \\
=c(\cSd)^{-1} \left( c_{k-\ell+1}(\cSd) + \cdots + c_k(\cSd)\right)
\end{multline*}
as stated.
\end{proof}

\subsection{}\label{ss:onerowres}
The information obtained in Corollary~\ref{cor:SigmaSrow} suffices by Theorem~\ref{thm:mainred}
to complete the computation of $d_{n,r,S}$. Here $S$ consists of $\ell$ entries in a row; 
we denote the number $d_{n,r,S}$ by $d_{n,r|\ell}$, since the 
row and the location of the entries in that row are immaterial. 

\begin{theorem}\label{thm:onerow}
With notation as above,
\begin{equation}\label{eq:finalonerow}
d_{n,r|\ell}=\int_{G(n-r,n)} c(\cSd)^{n-1} \sum_{i=0}^{n-r-\ell} c_i(\cSd)\quad.
\end{equation}
\end{theorem} 

\begin{proof}
Corollary~\ref{cor:SigmaSrow} yields
\begin{align*}
c(\cSd)^n\left(1-\Sigma_{n,r,S}\right)
&=c(\cSd)^n\left(1-c(\cSd)^{-1} \cap \left( c_{k-\ell+1}(\cSd) 
+ \cdots + c_k(\cSd)\right)\right) \\
&=c(\cSd)^{n-1} \left(c(\cSd) - \left( c_{k-\ell+1}(\cSd) 
+ \cdots + c_k(\cSd)\right)\right) \\
&=c(\cSd)^{n-1} \left(c_0(\cSd) + \cdots + c_{k-\ell}(\cSd)\right)\quad,
\end{align*}
which gives the statement by Theorem~\ref{thm:mainred}.
\end{proof}

\begin{example}\label{ex:exarow}
$\bullet$ $d_{n,r|0}=\int_{G(k,n)} c(\cSd)^n = \deg \sigma_{n,r}$ 
(Lemma~\ref{lem:degcom}).

$\bullet$ For $\ell=1$, \eqref{eq:finalonerow} gives
\begin{align*}
\int_{G(n-r,n)} c(\cSd)^{n-1} \sum_{i=0}^{n-r-1} c_i(\cSd)
&=\int_{G(n-r,n)} c(\cSd)^n - \int_{G(n-r,n)} c(\cSd)^{n-1}\cdot c_{n-r}(\cSd)\\
&=\int_{G(n-r,n)} c(\cSd)^n - \int_{G(n-r,n-1)} c(\cSd)^{n-1}\\
&=\deg \sigma_{n,r}-\deg \sigma_{n-1,r-1} \\
&=\prod_{i=0}^{n-r-1} \frac{\binom{n+i}{n-r}}{\binom{n-r+i}{n-r}} 
-\prod_{i=0}^{n-r-1} \frac{\binom{n-1+i}{n-r}}{\binom{n-r+i}{n-r}} 
\quad.
\end{align*}
Therefore, the multiplicity of $\sigma_{n,r}$ at a matrix of rank~$1$ is 
the degree of $\sigma_{n-1,r-1}$. This is clear for e.g.,~the determinant hypersurface
$\sigma_{n,n-1}$, since the tangent cone to $\sigma_{n,n-1}$ at a rank-$1$ matrix
is a cone over the determinant $\sigma_{n-1,n-2}$. 
In general, this degree computation
can be used to show that the tangent cone to $\sigma_{n,r}$ at a rank-$1$ matrix
is a cone over $\sigma_{n-1,r-1}$.

$\bullet$ $d_{n,r|n-r}=\int_{G(n-r,n)} c(\cSd)^{n-1}$. This degree
may also be evaluated explicitly, yielding
$
\prod_{i=0}^{n-r-2} \frac{\binom{n+i}{n-r}}{\binom{n-r+i}{n-r}}
$.

$\bullet$ For $\ell>n-r$, $d_{n,r|\ell}=0$. Indeed, $\sum_{i=0}^{n-r-\ell} c_i(\cSd)=0$ trivially
in this case. This fact was already mentioned in Remark~\ref{rem:deq0}.
\qede\end{example}

The following Macaulay2 script using Schubert2 (\cite{S2}):
\begin{verbatim}
needsPackage("Schubert2")
onerow = (n,r,l) -> (
    G = flagBundle({n-r,r});
    (S,Q) = bundles G;
    Sd = dual S;
    integral(chern(Sd)^(n-1)*sum(0..n-r-l, i-> chern(i,Sd)))
)
\end{verbatim}
produces a function {\tt onerow\/} which evaluates~\eqref{eq:finalonerow} given the
input of the dimension~$n$ of the ambient space, the rank $r$, and the number $\ell$
of entries of $S$ (which must all be in the same row).

\begin{example} 
The degrees of the joins of the locus of $7\times 7$ matrices of rank $\le r$ with the
space spanned by $3$ entries in the top row are as follows:
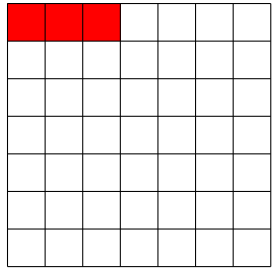
\[
\begin{tikzpicture}[baseline=(current  bounding  box.center)]
\path [fill=red] (0,0) rectangle (.5,-.5);
\path [fill=red] (.5,0) rectangle (1,-.5);
\path [fill=red] (1,0) rectangle (1.5,-.5);
\draw (0,0) --(0,-3.5);
\draw (.5,0) --(.5,-3.5);
\draw (1,0) --(1,-3.5);
\draw (1.5,0) --(1.5,-3.5);
\draw (2,0) --(2,-3.5);
\draw (2.5,0) --(2.5,-3.5);
\draw (3,0) --(3,-3.5);
\draw (3.5,0) --(3.5,-3.5);
\draw (0,0) --(3.5,0);
\draw (0,-.5) --(3.5,-.5);
\draw (0,-1) --(3.5,-1);
\draw (0,-1.5) --(3.5,-1.5);
\draw (0,-2) --(3.5,-2);
\draw (0,-2.5) --(3.5,-2.5);
\draw (0,-3) --(3.5,-3);
\draw (0,-3.5) --(3.5,-3.5);
\end{tikzpicture}
\quad\quad\quad\quad
d_{7,r|3}=\begin{cases}
896 & r=1 \\
15582 & r=2 \\
11172 & r=3 \\
490 & r=4 \\
0 & r=5 \\
0 & r=6 \\
0 & r=7
\end{cases}
\]
These may be computed quickly using the script given above. For example,
\begin{verbatim}
i2 : time onerow(7,3,3)
     -- used 0.00669973 seconds
o2 = 11172
\end{verbatim}
on our computing equipment. (By comparison, a verification of
the table for $n=7$, $\ell=3$ shown above took nearly twenty minutes on the
same computer, running a more straightforward algorithm for degrees of 
projections in Macaulay2.)
\qede\end{example}

\subsection{}
By symmetry, $d_{n,r,S} = d_{n,r,S^\dagger}$, where $S^\dagger$ is the transpose 
configuration of entries from~$S$. Thus, \eqref{eq:finalonerow} must also be the degree 
of the projection of the rank-$r$ locus from a set of entries lying in the same {\em column.\/} 
Performing an independent computation for this number from this different viewpoint leads
to an identity in the Chow ring of the Grassmannian.

Let $S$ consist of $\ell$ entries in one column, e.g.:
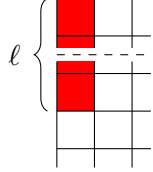
\begin{center}
\begin{tikzpicture}
\path [fill=red] (0,0) rectangle (.5,-.66);
\path [fill=red] (0,-.84) rectangle (.5,-1.5);
\draw (0,0) --(0,-.66);
\draw (0,-.84) --(0,-2.25);
\draw (.5,0) --(.5,-.66);
\draw (.5,-.84) --(.5,-2.25);
\draw (1,0) --(1,-.66);
\draw (1,-.84) --(1,-2.25);
\draw (0,0) --(1.25,0);
\draw (0,-.5) --(1.25,-.5);
\draw [dashed] (0,-.75) --(1.25,-.75);
\draw (0,-1) --(1.25,-1);
\draw (0,-1.5) --(1.25,-1.5);
\draw (0,-2) --(1.25,-2);
	\draw [decorate,decoration={brace,amplitude=5pt},
			xshift=0pt, yshift=0pt]
			(-.125,-1.5) -- (-.125,-0)
			node[black,midway,xshift=-12.5pt,yshift=0pt]
				{$\ell$};
\end{tikzpicture}
\end{center}
We use notation as above: $L_S$ is the span of the entries of $S$,
$\hL_S$ denotes $G(k,n)\times L_S\cong G(k,n)\cong \Pbb^{\ell-1}$, and $\htau_{n,k}$
is the standard resolution of $\tau_{n,k}=\sigma_{n,n-k}$. The Grassmann class 
$\Sigma_{n,r,S}$ (Definition~\ref{def:GrassS}) needed in order to apply 
Theorem~\ref{thm:mainred} is obtained from $s(\hL_S\cap \htau_{n,k},\hL_S)^{\cO(-H)}$, 
where $H$ is the pull-back of the hyperplane class from the $\Pbb^{\ell-1}$ factor.

\begin{lemma}\label{lem:onecoldes}
$\hL_S\cap \htau_{n,k}\cong G(k,n-1)\times \Pbb^{\ell-1}$.
\end{lemma}

\begin{proof}
The intersection $\hL_S\cap \htau_{n,k}$ consists of pairs $(K,\varphi)$ such that
$K\subseteq \ker\varphi$. Since $\varphi$ is concentrated in one column, $\ker\varphi$
consists of a fixed hyperplane, and the statement follows, at least set-theoretically.
A simple coordinate computation shows that this description in fact holds scheme-theoretically.
\end{proof}

\begin{corol}\label{cor:onecolclass}
Let $S$ consist of $\ell$ entries in a column. As a class in $\hL_S$,
\begin{equation}\label{eq:onecol}
s(\hL_S\cap \htau_{n,k}, \hL_S)^{\cO(-H)} = c(\cO(-H))^{-1} c(\cSd\otimes \cO(-H))^{-1} 
c_k(\cSd)\quad.
\end{equation}
\end{corol}

\begin{proof}
This follows from Lemma~\ref{lem:onecoldes} and Lemma~\ref{lem:twifacts} \eqref{pt1},
since $G(k,n-1)\times \Pbb^{\ell-1}$ is the zero-scheme of a regular section of $\cSd$.
\end{proof}

The reader should compare~\eqref{eq:onecol} with the similar but different 
expression~\eqref{eq:onerow} for $S$ concentrated in a row.

\begin{corol}\label{cor:SigmaScol}
Let $S$ consist of $\ell$ entries in a column. Then
\begin{equation}\label{eq:SigmaScol}
\Sigma_{n,r,S} = \sum_{i=0}^r \binom{\ell-1+n-r+i}{\ell-1} c_i(\cQd) c_{n-r}(\cSd)\quad.
\end{equation}
\end{corol}

\begin{proof}
Tensoring the dual of the universal sequence over $G(k,n)$ by $\cO(-H)$ gives an exact
sequence
\[
\xymatrix{
0 \ar[r] & \cQd\otimes \cO(-H) \ar[r] & \cO(-H)^{n-1} \ar[r] & \cSd\otimes \cO(-H)
\ar[r] & 0
}
\]
from which $c(\cSd\otimes \cO(-H))^{-1}=(1-H)^{-n} c(\cQd\otimes \cO(-H))$.
Substituting in~\eqref{eq:onecol} and applying the formula for the Chern class of a tensor 
product (\cite{85k:14004}, p.~55) gives
\[
s(\hL_S\cap \htau_{n,k}, \hL_S)^{\cO(-H)} = \sum_{i=0}^{n-k}
\frac{c_i(\cQd) c_k(\cSd)}{(1-H)^{k+i+1}}\quad.
\]
The Grassmann class is obtained by pushing this forward to $G(k,n)$, which amounts
to computing the coefficient of $H^{\ell-1}$, with the stated result.
\end{proof}

Applying Theorem~\ref{thm:mainred} gives the degree:

\begin{theorem}\label{thm:onecol}
If $S$ consists of $\ell$ elements in a column, then with $k=n-r$:
\begin{equation}\label{eq:finalonecol}
d_{n,r,S} = \int_{G(k,n)} c(\cSd)^n \left(1 - 
\sum_{i=0}^r \binom{\ell-1+k+i}{\ell-1} c_i(\cQd) c_k(\cSd)\right)
\quad.
\end{equation}
\end{theorem}

\subsection{}\label{ss:colexp}
As mentioned above, the expression obtained in Theorem~\ref{thm:onecol}
must agree with the expression given for $d_{n,r|\ell}$ in Theorem~\ref{thm:onerow}.
In other words, if $\Sigma'_{n,r|\ell}$, $\Sigma''_{n,r|\ell}$ denote the classes 
\eqref{eq:SigmaSrow}, \eqref{eq:SigmaScol} obtained for $S$ consisting of $\ell$ elements
of a row, resp., a column, then necessarily
\[
\int_{G(k,n)} c(\cSd)^n\cdot \Sigma'_{n,r|\ell} 
= \int_{G(k,n)} c(\cSd)^n\cdot \Sigma''_{n,r|\ell}\quad.
\]
This equality translates into identities in the Chow ring of the Grassmannian, for all choices 
of $n,r,\ell$. They are nontrivial, in the sense that in general the classes $\Sigma'_{n,r|\ell}$, 
$\Sigma''_{n,r|\ell}$ differ. In fact, even the dimensions of these classes need not 
agree.

\begin{example}
For $n=3$, $r=1$, $\ell=2$, $\Sigma'_{n,r|\ell}$ and $\Sigma''_{n,r|\ell}$ are classes in 
$G(2,3)\cong \Pbb^2$. The reader can verify that $\Sigma'_{3,1|2}$ is the hyperplane class 
$h$, while $\Sigma''_{3,1|2}$ equals $3h^2$.
\[
\begin{tikzpicture}[baseline=(current  bounding  box.center)]
\path [fill=red] (0,0) rectangle (1,-.5);
\draw (0,0) --(0,-1.5);
\draw (.5,0) --(.5,-1.5);
\draw (1,0) --(1,-1.5);
\draw (1.5,0) --(1.5,-1.5);
\draw (0,0) --(1.5,0);
\draw (0,-.5) --(1.5,-.5);
\draw (0,-1) --(1.5,-1);
\draw (0,-1.5) --(1.5,-1.5);
\end{tikzpicture}
\quad\quad
\Sigma'_{3,1|2}=h\quad\quad;\quad\quad
\begin{tikzpicture}[baseline=(current  bounding  box.center)]
\path [fill=red] (0,0) rectangle (.5,-1);
\draw (0,0) --(0,-1.5);
\draw (.5,0) --(.5,-1.5);
\draw (1,0) --(1,-1.5);
\draw (1.5,0) --(1.5,-1.5);
\draw (0,0) --(1.5,0);
\draw (0,-.5) --(1.5,-.5);
\draw (0,-1) --(1.5,-1);
\draw (0,-1.5) --(1.5,-1.5);
\end{tikzpicture}
\quad\quad
\Sigma''_{3,1|2}=3h^2\quad\quad
\]
With this notation, $c(\cSd)^3=(1+h+h^2)^3=1+3h+6h^2$, so that
\[
\int c(\cSd)^3\cdot \Sigma'_{3,1|2} = 3 = \int c(\cSd)^3\cdot \Sigma''_{3,1|2}
\]
as expected.
\qede\end{example}

We can use the equality of degrees to obtain information on
$d_{n,r|\ell}$. It is straightforward to write \eqref{eq:finalonecol} in terms of an intersection
degree in $G(k,n-1)$:
\begin{equation}\label{eq:onecolalt}
d_{n,r,S} = \prod_{i=0}^{k-1} \frac{\binom{n+i}k}{\binom{k+i}k} 
-\sum_{i=0}^{n-k-1} \binom{\ell-1+k+i}{\ell-1} \int_{G(k,n-1)} c(\cSd)^n c_i(\cQ'^\vee)\quad.
\end{equation}
where $\cQ'$ denotes the universal quotient bundle on $G(k,n-1)$.
The expression in~\eqref{eq:onecolalt} must evaluate $d_{n,r|\ell}$, and hence
agree with~\eqref{eq:finalonerow}. Formula~\eqref{eq:onecolalt} and the degree 
computation
\begin{equation}\label{eq:schubc}
\int_{G(k,n-1)} c(\cSd)^n \cdot c_i(\cQ')=
\frac{\binom n{k+1}\cdots \binom{n+k-1}{k+1}}{\binom {k+1}{k+1}\cdots \binom{2k}{k+1}}
\cdot\frac{\binom {n-k-1}i \binom{i+k-1}i }{\binom{2k+i}i}\quad,
\end{equation}
(which is likely well-known) yield the following explicit expression for $d_{n,r|\ell}$:
\begin{equation}\label{eq:exponerow}
d_{n,r|\ell} = \prod_{j=0}^{k-1} \frac{\binom{n+j}k}{\binom{k+j}k}
-\prod_{j=0}^{k-1} \frac{\binom{n+j}{k+1}}{\binom{k+1+j}{k+1}}
\sum_{i=0}^{n-k-1} (-1)^i \binom{\ell-1+k+i}{\ell-1} 
\frac{\binom {n-k-1}i \binom{i+k-1}i }{\binom{2k+i}i}
\end{equation}
for $\ell\le k=n-r$. (The script given in~\S\ref{ss:onerowres} confirms this expression
for dozens of examples.) A reader more 
versed in Schubert calculus than this writer could verify this formula rigorously
(by checking \eqref{eq:schubc})
or provide more directly an explicit expression computing~\eqref{eq:finalonerow}.

\subsection{}\label{ss:benzene}
Here is a table displaying $d_{a+b,a|b-2}$ for $a=1,\dots,10$ and $b=2,3$:
\begin{center}
  \begin{tabular}{ | r || r | r | r | r | r | r | r | r | r | r |}
    \hline
     & $a=1$ & 2 & 3 & 4 & 5 & 6 & 7 & 8 & 9 & 10\\ \hline \hline
    $b=2$ & 6 & 20 & 50 & 105 & 196 & 336 & 540 & 825 & 1210 & 1716 \\ \hline
    3 & 19 & 155 & 805 & 3136 & 9996 & 27468 & 67320 & 150645 & 313027 & 611611 \\
    \hline
  \end{tabular}
\end{center}
This table matches the one found on page 188 of~\cite{benzenoids1988kekule} 
for the values of $K(D^j(a,b))$, the number of {\em Kekul\'e structures for benzenoid
hydrocarbons\/} in an `oblate pentagon' configuration:
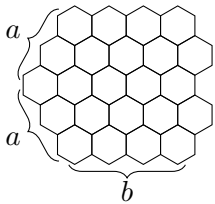
\begin{center}
\begin{tikzpicture}
  \begin{scope}[%
every node/.style={anchor=west,
regular polygon,
regular polygon rotate = 30, 
regular polygon sides=6,
draw,
minimum width=15pt,
outer sep=0,
},
      transform shape]
    \node (A) {};
    \node (B) at (A.side 5) {};
    \node (C) at (B.side 5) {};
    \node (D) at (C.side 5) {};
  \end{scope}
  \begin{scope}[%
every node/.style={anchor=north,
regular polygon,
regular polygon rotate = 30, 
regular polygon sides=6,
draw,
minimum width=15pt,
outer sep=0,
},
      transform shape]
    \node ({EA}) at (A.corner 3) {};
    \node ({FA}) at ({EA}.corner 3) {};
    \node ({GA}) at ({FA}.corner 5) {};
    \node ({HA}) at ({GA}.corner 5) {};
    \node ({EB}) at (B.corner 3) {};
    \node ({FB}) at ({EB}.corner 3) {};
    \node ({GB}) at ({FB}.corner 5) {};
    \node ({HB}) at ({GB}.corner 5) {};
    \node ({EC}) at (C.corner 3) {};
    \node ({FC}) at ({EC}.corner 3) {};
    \node ({GC}) at ({FC}.corner 5) {};
    \node ({HC}) at ({GC}.corner 5) {};
    \node ({ED}) at (D.corner 3) {};
    \node ({FD}) at ({ED}.corner 3) {};
    \node ({GD}) at ({FD}.corner 5) {};
    \node ({HD}) at ({GD}.corner 5) {};
    \node (E) at (D.corner 5) {};
    \node (F) at (E.corner 3) {};
    \node (G) at (F.corner 5) {};
  \end{scope}
	\draw [decorate,decoration={brace,amplitude=5pt},
			xshift=0pt, yshift=0pt]
			(1.7,-1.85) -- (0.15,-1.85)
			node[black,midway,yshift=-10pt]
				{$b$};
	\draw [decorate,decoration={brace,amplitude=5pt},
			xshift=0pt, yshift=0pt]
			(0.02,-1.8) -- (-.5,-.85)
			node[black,midway,xshift=-10pt,yshift=-5pt]
				{$a$};
	\draw [decorate,decoration={brace,amplitude=5pt},
			xshift=0pt, yshift=0pt]
			(-.5,-.72) -- (0.02,.22)
			node[black,midway,xshift=-10pt,yshift=5pt]
				{$a$};
\end{tikzpicture}
\end{center}
(For $b>3$, the values of $K(D^j(a,b))$ and $d_{a+b,a|b-2}$ as above do not match.)
Several other series of numbers of Kekul\'e structures may be expressed as numbers
$d_{n,r|\ell}$. For example, the second row in the above table is one of four nonzero
possibilities for $d_{n,n-3|\ell}$:
\begin{center}
  \begin{tabular}{ | r || r | r | r | r | r | r | r | r |}
    \hline
     & $n=3$ & 4 & 5 & 6 & 7 & 8 & 9 & 10\\ \hline \hline
    $\ell=0$ & 1 & 20 & 175 & 980 & 4116 & 14112 & 41580 & 108900 \\ \hline
    1 & 1 & 19 & 155 & 805 & 3136 & 9996 & 27468 & 67320 \\ \hline
    2 & 1 & 16 & 110 & 490 & 1666 & 4704 & 11592 & 25740 \\ \hline
    3 & 1 & 10 & 50 & 175 & 490 & 1176 & 2520 & 4950 \\
    \hline
  \end{tabular}
\end{center}
These agree with the numbers of Kekul\'e structures for benzenoids forming `$5$-tier 
dihedral hexagons, oblate pentagons, intermediate pentagons', and `$4$-tier
centrosymmetric hexagons' (\cite{benzenoids1988kekule}, p.~167 no.~1, 2, 3, 
and p.~166 no.~1, respectively). We are not aware of any direct connection between 
results in the extensive literature on Kekul\'e structures and the enumerative geometry 
of varieties expressing rank conditions. Is every degree $d_{n,r,S}$ the number of 
Kekul\'e structures for a benzenoid hydrocarbon molecule? The degrees
$\deg\sigma_{n,r}$ are all Kekul\'e numbers, for dihedral hexagons: with notation as
on p.~108 of~\cite{benzenoids1988kekule}, $K\{O(a,b)\}$ equals $d_{a+b,b|0}=\deg
\sigma_{a+b,b}$.
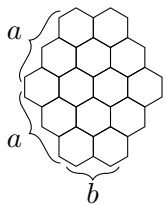
\begin{center}
\begin{tikzpicture}
  \begin{scope}[%
every node/.style={anchor=west,
regular polygon,
regular polygon rotate = 30, 
regular polygon sides=6,
draw,
minimum width=15pt,
outer sep=0,
},
      transform shape]
    \node (A) {};
    \node (B) at (A.side 5) {};
  \end{scope}
  \begin{scope}[%
every node/.style={anchor=north,
regular polygon,
regular polygon rotate = 30, 
regular polygon sides=6,
draw,
minimum width=15pt,
outer sep=0,
},
      transform shape]
    \node ({EA}) at (A.corner 3) {};
    \node ({FA}) at ({EA}.corner 3) {};
    \node ({GA}) at ({FA}.corner 5) {};
    \node ({HA}) at ({GA}.corner 5) {};
    \node ({EB}) at (B.corner 3) {};
    \node ({FB}) at ({EB}.corner 3) {};
    \node ({GB}) at ({FB}.corner 5) {};
    \node ({HB}) at ({GB}.corner 5) {};
    \node ({C}) at ({B}.corner 5) {};
    \node ({D}) at ({C}.corner 3) {};
    \node ({E}) at ({C}.corner 5) {};
    \node ({F}) at ({D}.corner 5) {};
  \end{scope}
	\draw [decorate,decoration={brace,amplitude=5pt},
			xshift=0pt, yshift=0pt]
			(.8,-1.85) -- (0.1,-1.85)
			node[black,midway,yshift=-10pt]
				{$b$};
	\draw [decorate,decoration={brace,amplitude=5pt},
			xshift=0pt, yshift=0pt]
			(0.02,-1.8) -- (-.5,-.85)
			node[black,midway,xshift=-10pt,yshift=-5pt]
				{$a$};
	\draw [decorate,decoration={brace,amplitude=5pt},
			xshift=0pt, yshift=0pt]
			(-.5,-.72) -- (0.02,.22)
			node[black,midway,xshift=-10pt,yshift=5pt]
				{$a$};
\end{tikzpicture}
\end{center}

There are several other apparent cameo appearances of these numbers in the literature,
which it would be interesting to explain more conceptually. The $\ell=0$ row in the above 
table (i.e., $\deg \sigma_{n, n-3}$) reproduces
the dimension $\dim V^{(n-3)}$ of the {\em distinguished module of dimension~$20$
of a Lie algebra in the subexceptional series\/} as per the $a=2$ case in Theorem~7.2 
of~\cite{MR2204753}. The $\ell=3$ row agrees with the dimension $H_5(w)$ of the
{\em space of semiinvariants of weight $w$,\/} $k[\text{Mat}(2,5)]^{SL(2)}_{(w)}$,
as computed in~\cite{MR2004218}, p.~238.


\section{Building $S$ from blocks}\label{blocks}

\subsection{}\label{ss:conessplayed}
In this section we will use the information obtained in Corollary~\ref{cor:SigmaSrow} 
and Corollary~\ref{cor:SigmaScol} to compute $d_{n,r,S}$ in the more interesting 
cases in which $S$ is built from entries in separate rows and columns. For example,
this will give $d_{n,r,S}$ when $S$ consists of a collection of entries $(i_s,j_s)$ such
that all $j_s$ are distinct, i.e., such that no two entries in $S$ are in the same column;
this will in particular cover the case in which $S$ is a subset of the diagonal of a matrix.

In general, assume $S$ (possibly after a permutation of rows and columns) consists of two blocks
$S'$, $S''$ with no overlapping rows or columns:
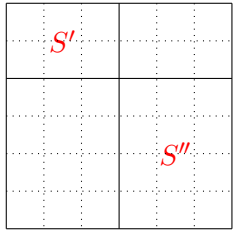
\begin{center}
\begin{tikzpicture}
\node at (.75,-.5) [red] {$S'$};
\node at (2.25,-2) [red] {$S''$};
\draw (0,0) --(0,-3);
\draw [dotted] (.5,0) --(.5,-3);
\draw [dotted] (1,0) --(1,-3);
\draw (1.5,0) --(1.5,-3);
\draw [dotted] (2,0) --(2,-3);
\draw [dotted] (2.5,0) --(2.5,-3);
\draw (3,0) --(3,-3);
\draw (0,0) --(3,0);
\draw [dotted] (0,-.5) --(3,-.5);
\draw (0,-1) --(3,-1);
\draw [dotted] (0,-1.5) --(3,-1.5);
\draw [dotted] (0,-2) --(3,-2);
\draw [dotted] (0,-2.5) --(3,-2.5);
\draw (0,-3) --(3,-3);
\end{tikzpicture}
\end{center}
We will prove a relation at the level of Grassmann classes:
\[
(1-\Sigma_{n,r,S}) = (1-\Sigma_{n,r,S'}) (1-\Sigma_{n,r,S''})
\]
(Theorem~\ref{thm:Sigblocks}) which will allow us to analyze the cases mentioned 
above; in fact, it will lead to a direct generalization of the results obtained 
in~\S\ref{singlerow}.
This multiplicative property organizes the degrees $d_{n,r,S}$, but is invisible at the 
numerical level.

The proof of Theorem~\ref{thm:Sigblocks} relies on more sophisticated properties 
of Segre classes, which we state in this subsection. 

First, we need to deal with a generalization of the join construction. Consider two
disjoint subspaces $\Pbb^m$, $\Pbb^{M-m-1}$ of a projective space $\Pbb^M$,
and let $I$ be a homogeneous ideal, defining a subscheme $Z$ of $\Pbb^m$. 
Let $\vZ\subseteq \Pbb^M$ be the subscheme of $\Pbb^M$ defined by a set
of generators of $I$. 
Geometrically, $\vZ$ is the join of $Z$ and $\Pbb^{M-m-1}$ in $\Pbb^M$ (i.e., 
the cone over $Z$ with vertex $\Pbb^{M-m-1}$).

Since $Z$ and $\Pbb^{M-m-1}$ are disjoint, the projection of $Z$ from $\Pbb^{M-m-1}$ 
is isomorphic to~$Z$ itself, and it follows that the degree of $\vZ$ equals the degree $Z$.
Segre classes are a more refined invariant than degree, so a relation between
$s(Z,\Pbb^m)$ and $s(\vZ,\Pbb^M)$ is bound to be subtler. For example, the Segre
class of $\vZ$ depends on the specific ideal $I$ chosen to define~$Z$.

\begin{example}
Let $m=1$ and $M=2$, so we are dealing with the simple case of joining a point
in the plane with a subscheme of a line. The ideals $(x_0,x_1)$ and $(x_0^2,x_0 x_1)$
define the same subscheme $Z$ in $\Pbb^1$ (i.e., a reduced point); but the Segre 
classes of the corresponding joins, defined by $(x_0,x_1)$ and $(x_0^2,x_0 x_1)$ 
in $\Pbb^2$, differ (due to the embedded component at the origin in the second case).
\qede\end{example}

In fact we need to deal with a somewhat more general situation, for which $Z$ is a 
subscheme of a product $V\times \Pbb^m$. 
We assume that $Z$ is defined by a section of a vector bundle 
$\cE\otimes \cO(H)$, where $\cE$ is the pull-back of a 
vector bundle from $V$, and $H$ denotes the hyperplane class. (This condition will
be automatically verified for all the examples we will consider.) For $M>m$, we
let $\vZ= Z \vee (V\times \Pbb^{M-m-1})$ be the subscheme of $V\times \Pbb^M$
defined by the `pull-back' of the section to $V\times \Pbb^M$.
Since this join operation $\vee$ clearly preserves rational equivalence, we can
extend it to the Chow group and define a homomorphism $\alpha \mapsto
\alpha\vee (V\times \Pbb^{M-m-1})$ from $A_*Z$ to $A_*\vZ$.

\begin{lemma}\label{lem:scone}
With the above notation,
\begin{multline}\label{eq:cones}
s(\vZ, V\times \Pbb^M)^{\cO(-H)} \\ 
= s(Z,V\times \Pbb^m)^{\cO(-H)}\vee 
(V\times\Pbb^{M-m-1}) + (1+H+H^2+\cdots) \cap [V\times \Pbb^{M-m-1}]\quad.
\end{multline}
\end{lemma} 

This is a particular case of a rationality result for Segre classes of ind-schemes,
and its proof is discussed elsewhere (\cite{Segrat}).
This result will allow us to control the behavior of the relevant classes as the chosen
set $S$ of entries increases in size. In fact, \eqref{eq:cones} is our main reason for 
choosing to work with the class introduced in Definition~\ref{def:twisSc}: a relation 
for ordinary Segre classes may be obtained from \eqref{eq:cones}, but is more 
complicated. As a rule, formulas for the twisted classes in the concrete situations that 
are considered in this note are simpler than formulas for ordinary Segre classes
(this is already the case for the formula obtained in Corollary~\ref{cor:onerowclass}).

A second result we will need concerns {\em splayed\/} subschemes. Two subschemes
$X$, $Y$ of a nonsingular variety $V$ are {\em splayed\/} if at each point of their intersection
there are analytic coordinates $(x_1,\dots, x_a,y_1,\dots, y_b)$ for $V$ such that $X$ may be 
defined by an ideal generated by functions in the coordinates $x_i$ and $Y$ by an ideal 
generated by functions in the coordinates~$y_j$.

\begin{lemma}\label{lem:splay}
Let $X$ and $Y$ be splayed in $V$, and let $\cL$ be a line bundle on $V$. Then 
\begin{equation}\label{eq:splayed}
s(X\cap Y, V)^\cL = c(\cL)\cap (s(X,V)^\cL \cdot s(Y,V)^\cL)
\end{equation}
in the Chow group of $X\cap Y$.
\end{lemma}

\begin{proof}
By Lemma~3.1 in~\cite{1406.1182},
\[
s(X\cap Y,V) = s(X,V)\cdot s(Y,V)\quad.
\] 
The stated formula follows immediately from this result, Definition~\ref{def:twisSc},
and the evident fact that $\otimes_V$ preserves intersection products in $V$.
\end{proof}

\subsection{}
Now assume $S$ consists of two blocks $S'$, $S''$, as in~\S\ref{ss:conessplayed}.
Since $S'\subseteq S$, we have an embedding $L_{S'}\subseteq L_{S}$, and
hence an embedding $\hL_{S'}\subseteq \hL_S$. Similarly, $\hL_{S''}\subseteq \hL_S$.
For any $Z'\subseteq \hL_{S'}$, we can consider the join $\vZ'$ of $Z'$ and 
$\hL_{S''}$ in $\hL_S$, and similarly we let $\vZ''$ be the join of $Z''$ and $\hL_{S'}$,
for every subscheme $Z''\subseteq \hL_{S''}$. 

\begin{lemma}\label{lem:blockstr}
Let $Z'=\hL'_{S'}\cap \htau_{n,k}$, $Z''=\hL'_{S''}\cap \htau_{n,k}$. Then with notation
as above:
\begin{itemize}
\item Both $\vZ'$, $\vZ''$ are zeros of sections of bundles $\cE\otimes \cO(H)$,
with $\cE$ a pull-back from $G(k,n)$;
\item $\hL_S\cap \htau_{n,k} = \vZ' \cap \vZ''$ (scheme-theoretically); and
\item $\vZ'$ and $\vZ''$ are splayed.
\end{itemize}
\end{lemma}

\begin{proof}
Denote elements of $\hL_S=G(k,n)\times L_S$ by pairs $(C,A)$, where $C=(c_{ij})$ is a
$k\times n$ matrix whose rows span the given element of $G(k,n)$, and $A$ is
an $n\times n$ matrix with entries in $S$. The condition that $(C,A)\in \hL_S\cap
\htau_{n,k}$ means that the row-span of $C$ is contained in the kernel of $A$;
thus, the ideal defining $\hL_S\cap \htau_{n,k}$ is generated by the entries of $A\cdot C^t$.

Under the current assumption on $S$, $A$ is a block matrix:
\[
A=\left(
\begin{array}{c|c}
A'  & 0 \\ \hline
0 & A'' 
\end{array}\right)\quad;
\]
and we can split $C^t$ horizontally into matrices $C'$, $C''$ with entries $c_{ij}$ 
with $j\in J'$, $j\in J''$ resp., so that the product $A\cdot C^t$ is in fact
\[
\left(
\begin{array}{c|c}
A'  & 0 \\ \hline
0 & A'' 
\end{array}\right)
\cdot
\left(
\begin{array}{c}
C' \\ \hline
C''
\end{array}\right)
=
\left(
\begin{array}{c}
A'\cdot C' \\ \hline
A''\cdot C''
\end{array}\right)\quad.
\]
With this notation, $\vZ'$ is defined by the vanishing of the entries of $A'\cdot C'$,
and $\vZ''$ by the vanishing of the entries of $A''\cdot C''$. 

The statement follows. The first point is clear.
For the second point, the ideal of $\hL_S\cap \htau_{n,k}$ is indeed the sum of
the ideals of $\vZ'$ and $\vZ''$. The third point holds since generators for the ideals 
of $\vZ'$, $\vZ''$ are expressed in different sets of coordinates (in every local chart).
\end{proof}

\begin{theorem}\label{thm:Sigblocks}
Assume that $S$ consists of blocks $S_1,\dots, S_m$. Then
\[
\left(1-\Sigma_{n,r,S}\right) = \prod_{i=1}^m \left(1-\Sigma_{n,r,S_i}\right)
\]
in $G(n-r,n)$.
\end{theorem}

\begin{proof}
By an immediate induction it suffices to prove the case $m=2$, so we may assume
$S$ is the union of two blocks $S'$, $S''$ as above. The needed formula follows then
from Lemma~\ref{lem:blockstr} and the technical 
Lemmas~\ref{lem:scone} and~\ref{lem:splay}. Indeed, since $\hL_S\cap \htau_{n,k}$
is the splayed intersection of $\vZ'$ and $\vZ''$,
\[
s(\hL_S\cap \htau_{n,k},\hL_S)^{\cO(-H)} = (1-H)\cdot s(\vZ',\hL_S)^{\cO(-H)}\cdot 
s(\vZ'',\hL_S)^{\cO(-H)}
\]
by Lemma~\ref{lem:splay};
and $s(\vZ',\hL_S)^{\cO(-H)}$, $s(\vZ'',\hL_S)^{\cO(-H)}$ may be expressed in terms
of joins by using Lemma~\ref{lem:scone}, so that
\begin{multline*}
s(\hL_S\cap \htau_{n,k},\hL_S)^{\cO(-H)} 
=(1-H)\left(s(Z',\hL_{S'})^{\cO(-H)}\vee [\hL_{S''}]+(1-H)^{-1}\cap [\hL_{S''}]\right)\\
\cdot \left(s(Z'',\hL_{S''})^{\cO(-H)}\vee [\hL_{S'}]+(1-H)^{-1}\cap [\hL_{S'}]\right)\quad.
\end{multline*}
By definition of Grassmann class of a set of entries (Definition~\ref{def:GrassS}),
\[
s(Z',\hL_{S'})^{\cO(-H)} = \left(\alpha'_0 + \alpha'_1 H+\cdots + \alpha'_{s'-2} H^{s'-2} 
+ \Sigma_{n,r,S'} H^{s'-1}\right)\cap [\hL_{S'}]
\]
where $s'$ is the size of $S'$ and the $\alpha'_i$ are (pull-backs of) classes in $G(k,n)$;
and similarly for~$Z''$. These same expressions hold for the joins (but capping against
$[\hL_S]$). The Grassmann class $\Sigma_{n,r,S}$ is therefore given by the coefficient of 
$H^{s'+s''-1}$ in
\[
(1-H)\left(\alpha'_0+\cdots+ \Sigma_{n,r,S'} H^{s'-1} + \frac{H^{s'}}{1-H}\right)
\cdot \left(\alpha''_0+\cdots+ \Sigma_{n,r,S''} H^{s''-1} + \frac{H^{s''}}{1-H}\right)\quad.
\]
This is
$-\Sigma_{n,r,S'}\Sigma_{n,r,S''} + \Sigma_{n,r,S'} + \Sigma_{n,r,S''}$,
and the statement follows immediately.
\end{proof}

\subsection{}\label{ss:manyrows1}
As an illustration of the use of Theorem~\ref{thm:Sigblocks}, we consider the case in
which {\em no two of the $s$ entries in $S$ are in the same column.\/}
\begin{figure}
\begin{tikzpicture}[baseline=(current  bounding  box.center)]
\path [fill=red] (0,0) rectangle (.5,-.5);
\path [fill=red] (1,0) rectangle (1.5,-.5);
\path [fill=red] (3,0) rectangle (3.5,-.5);
\path [fill=red] (1.5,-1) rectangle (2,-1.5);
\path [fill=red] (2,-1.5) rectangle (2.5,-2);
\path [fill=red] (2.5,-1.5) rectangle (3,-2);
\path [fill=red] (.5,-2.5) rectangle (1,-3);
\draw (0,0) --(0,-3.5);
\draw (.5,0) --(.5,-3.5);
\draw (1,0) --(1,-3.5);
\draw (1.5,0) --(1.5,-3.5);
\draw (2,0) --(2,-3.5);
\draw (2.5,0) --(2.5,-3.5);
\draw (3,0) --(3,-3.5);
\draw (3.5,0) --(3.5,-3.5);
\draw (0,0) --(3.5,0);
\draw (0,-.5) --(3.5,-.5);
\draw (0,-1) --(3.5,-1);
\draw (0,-1.5) --(3.5,-1.5);
\draw (0,-2) --(3.5,-2);
\draw (0,-2.5) --(3.5,-2.5);
\draw (0,-3) --(3.5,-3);
\draw (0,-3.5) --(3.5,-3.5);
\end{tikzpicture}
\caption{}
\label{fig:d7r3211}
\end{figure}
We can describe $S$ as the union of the subsets~$S_i$ of entries in the $i$-th
row of the matrix; and we let $\ell_i$ be the cardinality of $S_i$. Thus $S$
consists of $\ell_1$ entries in the first row, $\ell_2$ entries in the second row, etc.,
subject to the condition that no two entries are in the same column. 
We have $\ell_1+\cdots+\ell_n=s\le n$.
The number $d_{n,r,S}$ clearly only depends on the numbers $\ell_1,\dots, \ell_n$,
not on the specific set $S$ of entries; so we can use the notation $d_{n,r,S}
=d_{n,r|\ell_1,\ell_2,\dots}$. Since permutations of the $\ell_i$'s do not affect the degree 
$d_{r,n,S}$, we can in fact list the $\ell_i$'s in any order and omit trailing 
zeros. Thus, this notation is compatible with the notation used in~\S\ref{singlerow}.
The degree for rank $r$ and the configuration shown in Figure~\ref{fig:d7r3211} can
be denoted $d_{7,r|3,2,1,1}$.

Of course, up to a permutation of the columns and rows, such a configuration consists of
blocks:
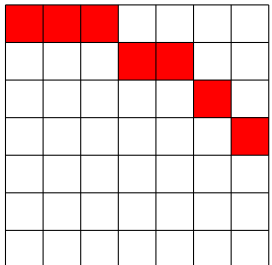
\begin{center}
\begin{tikzpicture}[baseline=(current  bounding  box.center)]
\path [fill=red] (0,0) rectangle (1.5,-.5);
\path [fill=red] (1.5,-.5) rectangle (2.5,-1);
\path [fill=red] (2.5,-1) rectangle (3,-1.5);
\path [fill=red] (3,-1.5) rectangle (3.5,-2);
\draw (0,0) --(0,-3.5);
\draw (.5,0) --(.5,-3.5);
\draw (1,0) --(1,-3.5);
\draw (1.5,0) --(1.5,-3.5);
\draw (2,0) --(2,-3.5);
\draw (2.5,0) --(2.5,-3.5);
\draw (3,0) --(3,-3.5);
\draw (3.5,0) --(3.5,-3.5);
\draw (0,0) --(3.5,0);
\draw (0,-.5) --(3.5,-.5);
\draw (0,-1) --(3.5,-1);
\draw (0,-1.5) --(3.5,-1.5);
\draw (0,-2) --(3.5,-2);
\draw (0,-2.5) --(3.5,-2.5);
\draw (0,-3) --(3.5,-3);
\draw (0,-3.5) --(3.5,-3.5);
\end{tikzpicture}
\end{center}
so that this case is covered by Theorem~\ref{thm:Sigblocks} and the computation
for individual rows carried out in~\S\ref{ss:singlerowlocus} and~\S\ref{ss:onerowres}.

\begin{corol}\label{cor:Sigmanyrows}
For $S$ described by $(\ell_1,\dots, \ell_n)$ as above,
\[
1-\Sigma_{n,r,S} = \prod_{i=1}^n
c(\cQd)  \left( c_0(\cSd) + \cdots + c_{k-\ell_i}(\cSd)\right)\quad.
\]
\end{corol}

This follows immediately from~Theorem~\ref{thm:Sigblocks} and
Corollary~\ref{cor:SigmaSrow}. Note that for $\ell_i = 0$, the contribution is
\[
c(\cQd)  \left( c_0(\cSd) + \cdots + c_k(\cSd)\right)
=c(\cQd) c(\cSd) = 1\quad,
\]
so that such rows do not affect the Grassmann class (as should be expected).

\begin{theorem}\label{thm:nocol}
\begin{equation}\label{eq:nocol}
d_{n,r|\ell_1,\dots,\ell_n} = \int_{G(n-r,n)} \prod_{i=1}^n \big(
c_0(\cSd)+ \cdots + c_{n-r-\ell_i}(\cSd)
\big)\quad.
\end{equation}
\end{theorem}

This statement generalizes directly the case in which $S$ is concentrated in one row,
i.e., Theorem~\ref{thm:onerow}. (Indeed, factors corresponding to $\ell_i=0$ reproduce
$c(\cSd)$.) Note that the pretty multiplicative structure underlying~\eqref{eq:nocol}
is invisible at the level of the degrees $d_{n,r|\ell_1,\dots,\ell_n}$. 

\begin{proof}
This now follows immediately from Theorem~\ref{thm:mainred} and 
Corollary~\ref{cor:Sigmanyrows}.
\end{proof}

It is straightforward to adapt the Macaulay2 script given in~\S\ref{ss:onerowres}
and implement the formula~\eqref{eq:nocol} given in Theorem~\ref{thm:nocol}:
\begin{verbatim}
needsPackage("Schubert2")
manyrows = (n,r,li) -> (
    m=#li;
    G = flagBundle({n-r,r});
    (S,Q) = bundles G;
    Sd = dual S;
    integral(chern(Sd)^(n-m)*product(0..m-1,
            j-> sum(0..n-r-li_j, i-> chern(i,Sd))))
)
\end{verbatim}
For example, \begin{verbatim} manyrows(7,3,{3,2,1,1})\end{verbatim} returns~$5990$,
the degree of the projection of the locus of $7\times 7$ matrices of rank $\le 3$ from 
the set $S$ depicted in Figure~\ref{fig:d7r3211}. This script took 0.023 seconds 
using our computing equipment to compile the data for all ranks for this set $S$:
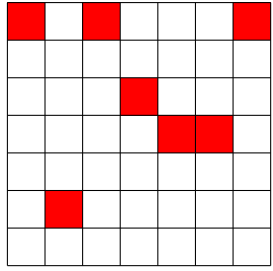
\[
\begin{tikzpicture}[baseline=(current  bounding  box.center)]
\path [fill=red] (0,0) rectangle (.5,-.5);
\path [fill=red] (1,0) rectangle (1.5,-.5);
\path [fill=red] (3,0) rectangle (3.5,-.5);
\path [fill=red] (1.5,-1) rectangle (2,-1.5);
\path [fill=red] (2,-1.5) rectangle (2.5,-2);
\path [fill=red] (2.5,-1.5) rectangle (3,-2);
\path [fill=red] (.5,-2.5) rectangle (1,-3);
\draw (0,0) --(0,-3.5);
\draw (.5,0) --(.5,-3.5);
\draw (1,0) --(1,-3.5);
\draw (1.5,0) --(1.5,-3.5);
\draw (2,0) --(2,-3.5);
\draw (2.5,0) --(2.5,-3.5);
\draw (3,0) --(3,-3.5);
\draw (3.5,0) --(3.5,-3.5);
\draw (0,0) --(3.5,0);
\draw (0,-.5) --(3.5,-.5);
\draw (0,-1) --(3.5,-1);
\draw (0,-1.5) --(3.5,-1.5);
\draw (0,-2) --(3.5,-2);
\draw (0,-2.5) --(3.5,-2.5);
\draw (0,-3) --(3.5,-3);
\draw (0,-3.5) --(3.5,-3.5);
\end{tikzpicture}
\quad\quad\quad\quad
d_{7,r|3,2,1,1}=\begin{cases}
887 & r=1 \\
13957 & r=2 \\
5990 & r=3 \\
35 & r=4 \\
0 & r=5 \\
0 & r=6 \\
0 & r=7
\end{cases}
\]

\subsection{}\label{ss:manyrows2}
It is clear from~\eqref{eq:nocol} that $d_{n,r|\ell_1,\dots,\ell_n}=0$ unless $\ell_i\le n-r$
for all $i$. (Geometrically, this signals that the projection contracts the locus, 
cf.~Remark~\ref{rem:deq0}.) We do not have a (conjectural) formula such 
as~\eqref{eq:exponerow} computing $d_{n,r,S}$ as a combination of binomial coefficients
in the more general case considered 
here; it is probably possible to obtain such a formula, or at least to conjecture
one based on numerical experimentation.
We will only deal explicitly with the case in which {\em no two elements of $S$ lie in the same
row or column,\/} i.e., all $\ell_i=0$ or $1$. Up to a permutation of rows and column, $S$ 
may be realized as a set of $s$ diagonal entries. 
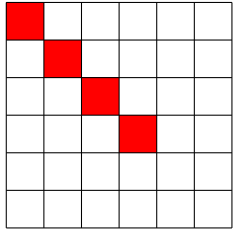
\begin{center}
\begin{tikzpicture}[baseline=(current  bounding  box.center)]
\path [fill=red] (0,0) rectangle (.5,-.5);
\path [fill=red] (.5,-.5) rectangle (1,-1);
\path [fill=red] (1,-1) rectangle (1.5,-1.5);
\path [fill=red] (1.5,-1.5) rectangle (2,-2);
\draw (0,0) --(0,-3);
\draw (.5,0) --(.5,-3);
\draw (1,0) --(1,-3);
\draw (1.5,0) --(1.5,-3);
\draw (2,0) --(2,-3);
\draw (2.5,0) --(2.5,-3);
\draw (3,0) --(3,-3);
\draw (0,0) --(3,0);
\draw (0,-.5) --(3,-.5);
\draw (0,-1) --(3,-1);
\draw (0,-1.5) --(3,-1.5);
\draw (0,-2) --(3,-2);
\draw (0,-2.5) --(3,-2.5);
\draw (0,-3) --(3,-3);
\end{tikzpicture}
\end{center}
The following result will also answer a question raised in~\cite{arXiv1310.1362}.

\begin{theorem}\label{thm:degdia}
For $s=1,\dots, n$,
\[
d_{n,r|1^s} = \sum_{j=0}^s \binom sj (-1)^j \deg \sigma_{n-j,r-j}
=\sum_{j=0}^s \binom sj (-1)^j \prod_{i=0}^{n-r-1} \frac{\binom{n-j+i}{n-r}}{\binom{n-r+i}{n-r}}
\]
\end{theorem}

\begin{proof}
We have $\ell_i=1$ for $s$ rows, and $\ell_i=0$ for the remaining $n-s$ rows.
According to Theorem~\ref{thm:nocol},
\begin{align*}
d_{n,r|1^s}&=\int_{G(n-r,n)} c(\cSd)^{n-s} \big(c(\cSd)-c_k(\cSd)\big)^s \\
&=\sum_{j=0}^s \binom sj (-1)^j \int_{G(n-r,n)} c(\cSd)^{n-j} c_k(\cSd)^j\quad.
\end{align*}
Since $c_k(\cSd)^j$ is the class of $G(n-r,n-j)$ in $G(n-r,n)$, this gives
\[
d_{n,r|1^s}=\sum_{j=0}^s \binom sj (-1)^j \int_{G(n-r,n-j)} c(\cSd)^{n-j}
=\sum_{j=0}^s \binom sj (-1)^j \deg\sigma_{n-j,r-j}\quad,
\]
by Lemma~\ref{lem:degcom}.
\end{proof}

Theorem~\ref{thm:degdia} confirms the equality stated in Theorem~5.4.7 
of~\cite{arXiv1310.1362}. In this statement, the equality is conditional on (and in fact
proven to be equivalent to) the following conjectural statement. For $x\in S$, 
let $S'=S\smallsetminus \{x\}$. Denote by $TC_xZ$ the tangent cone to $Z$ at $x$.
Conjecture~5.4.6 in~\cite{arXiv1310.1362} states:
\begin{quote}
Let $S$ be such that no two elements of $S$ lie in the same row or column and 
let $x \in S$. Then $TC_xJ(\sigma_{n,r},L_S) = J(TC_x \sigma_{n,r},L_{S'})$.
\end{quote}
(Here $J$ denotes join.)
As Theorem~\ref{thm:degdia} verifies the equality in Theorem~5.4.7 of~{\it loc.~cit.}, 
it implies this statement.

\begin{example}\label{ex:maxdia}
If $S$ consists of $s=(n-r)^2$ diagonal elements, then the projection of $\sigma_{n,r}$ 
from $L_S$ is dominant onto $\Pbb^{n^2-s-1}$. The projection is generically $d$-to-$1$, 
with
\[
d=d_{n,r|1^{(n-r)^2}}
=\sum_{j=0}^{(n-r)^2} \binom {(n-r)^2}j (-1)^j \prod_{i=0}^{n-r-1} 
\frac{\binom{n-j+i}{n-r}}{\binom{n-r+i}{n-r}}
\quad.
\]
Thus
\[
d=1,2,42, 24024, 701149020, \dots \quad\quad s\ge 1
\]
for $s=(n-r)^2$. Remarkably, this degree appears to be independent of $n$, for $n\ge s^2$:
for example, the projections of both $\sigma_{25,20}$ and $\sigma_{125,120}$ are
generically $701149020$-to-$1$ onto $\Pbb^{599}$, $\Pbb^{15599}$, respectively.
(We do not have a conceptual explanation for this experimental observation.)
\qede\end{example}

\subsection{}
Assume next that {\em no element of $S$ is on {\em both\/} the same row {\em and\/} the 
same column as other elements of $S$.\/} For example:
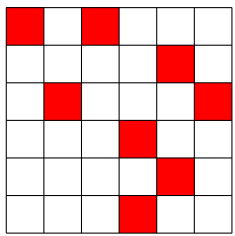
\begin{center}
\begin{tikzpicture}[baseline=(current  bounding  box.center)]
\path [fill=red] (.5,-1) rectangle (1,-1.5);
\path [fill=red] (2.5,-1) rectangle (3,-1.5);
\path [fill=red] (1.5,-1.5) rectangle (2,-2);
\path [fill=red] (1.5,-2.5) rectangle (2,-3);
\path [fill=red] (1,0) rectangle (1.5,-.5);
\path [fill=red] (0,0) rectangle (.5,-.5);
\path [fill=red] (2,-.5) rectangle (2.5,-1);
\path [fill=red] (2,-2) rectangle (2.5,-2.5);
\draw (0,0) --(0,-3);
\draw (.5,0) --(.5,-3);
\draw (1,0) --(1,-3);
\draw (1.5,0) --(1.5,-3);
\draw (2,0) --(2,-3);
\draw (2.5,0) --(2.5,-3);
\draw (3,0) --(3,-3);
\draw (0,0) --(3,0);
\draw (0,-.5) --(3,-.5);
\draw (0,-1) --(3,-1);
\draw (0,-1.5) --(3,-1.5);
\draw (0,-2) --(3,-2);
\draw (0,-2.5) --(3,-2.5);
\draw (0,-3) --(3,-3);
\end{tikzpicture}
\end{center}
The degrees $d_{n,r,S}$ may be computed for such configurations by using 
Theorem~\ref{thm:Sigblocks}. Indeed, up to permutations of rows and columns,
$S$ may be described as a block matrix with blocks consisting of subsets of
rows or subsets of columns. We describe $S$ by the lists $(\ell_1,\ell_2,\dots)$
of lengths of rows and $(m_1,m_2,\dots)$ of lengths of columns, and write
$d_{n,r| \ell_1,\dots | m_1,\dots}$ for the corresponding degrees. For example,
the degrees for the configuration shown above would  be denoted $d_{6,r|2,2|2,2}$.

\begin{theorem}\label{thm:mix}
\begin{multline*}
d_{n,r|\ell_1,\dots,\ell_a|m_1,\dots,m_b} = \int_{G(n-r,n)} c(\cSd)^{n-a} \\
\cdot \prod_{j=1}^a \left( \sum_{i=0}^{n-r-\ell_j}c_i(\cSd)\right)
\cdot\prod_{j=1}^b \left(1 - 
\sum_{i=0}^r \binom{m_j-1+n-r+i}{m_j-1} c_i(\cQd) c_{n-r}(\cSd)\right)
\quad.
\end{multline*}
\end{theorem}

\begin{example}\label{ex:mixed}
For the example shown above:
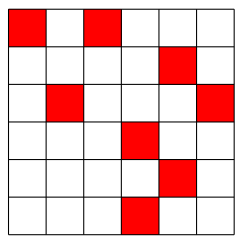
\[
\begin{tikzpicture}[baseline=(current  bounding  box.center)]
\path [fill=red] (.5,-1) rectangle (1,-1.5);
\path [fill=red] (2.5,-1) rectangle (3,-1.5);
\path [fill=red] (1.5,-1.5) rectangle (2,-2);
\path [fill=red] (1.5,-2.5) rectangle (2,-3);
\path [fill=red] (1,0) rectangle (1.5,-.5);
\path [fill=red] (0,0) rectangle (.5,-.5);
\path [fill=red] (2,-.5) rectangle (2.5,-1);
\path [fill=red] (2,-2) rectangle (2.5,-2.5);
\draw (0,0) --(0,-3);
\draw (.5,0) --(.5,-3);
\draw (1,0) --(1,-3);
\draw (1.5,0) --(1.5,-3);
\draw (2,0) --(2,-3);
\draw (2.5,0) --(2.5,-3);
\draw (3,0) --(3,-3);
\draw (0,0) --(3,0);
\draw (0,-.5) --(3,-.5);
\draw (0,-1) --(3,-1);
\draw (0,-1.5) --(3,-1.5);
\draw (0,-2) --(3,-2);
\draw (0,-2.5) --(3,-2.5);
\draw (0,-3) --(3,-3);
\end{tikzpicture}
\quad\quad\quad\quad
d_{6,r|2,2|2,2}=\begin{cases}
228 & r=1 \\
734 & r=2 \\
8 & r=3 \\
0 & r=4 \\
0 & r=5 \\
0 & r=6 
\end{cases}
\]
An explicit Macaulay2 computation shows that the projection of $\sigma_{6,3}$ 
($\dim_{6,3}=6^2-3^2-1=26$) from $L_S$ is a hypersurface of degree $4$ in 
$\Pbb^{27}$. Since $d_{6,r|2,2|2,2}=8$, the projection map must be
generically $2$-to-$1$ in this case.
\qede\end{example}

\subsection{}\label{ss:benzene2}
As in the case in which $S$ is concentrated in one row (cf.~\S\ref{ss:benzene}), 
the numbers $d_{n,r,S}$ computed in this section may have significance in other 
contexts, and it would be interesting to explore conceptual connections suggested
by these apparent coincidences. For example, the numbers
\[
d_{n,n-3|1^5} = 85, 295, 771, 1681, 3235, 5685, 9325, \dots \quad\quad (n\ge 5)
\]
form the {\em Ehrhart series for the matroid $K_4$\/} according to~\cite{oeis}.
The numbers $d_{n,n-3|1,1}$, $d_{n,n-3|1,2}$, $d_{n,n-3|2,2}$, etc.~all appear 
to be Kekul\'e numbers for certain benzenoids (\S\ref{ss:benzene}). For
example, 
\[
d_{n,n-3|2,3}=25, 65, 140, 266, 462, 750, 1155,\dots \quad\quad (n\ge 5)
\]
reproduce the Kekul\'e numbers for benzenoid hydrocarbons in the {\em chevron\/} 
configuration $Ch(2,3,n-3)$, cf.~\cite{benzenoids1988kekule}, p.~166.
The numbers
\[
d_{n,n-3|2|2} = 12, 60, 200, 525, 1176, 2352, 4320,\dots\quad\quad (n\ge 4)
\]
and
\[
d_{n,n-3|2|3} = 6, 20, 50, 105, 196, 336, 540, \dots \quad\quad (n\ge 4)
\]
are likewise Kekul\'e numbers (\cite{benzenoids1988kekule}, p.~233 \#11 and p.~165 I1).


\section{Increasing the complexity of $S$}\label{mixing}

\subsection{}
The cases studied in~\S\ref{ss:manyrows1} and~ff.~are still very special 
from the point of view of the initial motivation mentioned in the introduction. 
The smallest $S$ not covered by these cases is of the following type:
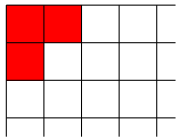
\begin{center}
\begin{tikzpicture}
\path [fill=red] (0,0) rectangle (.5,-.5);
\path [fill=red] (0,-0.5) rectangle (.5,-1);
\path [fill=red] (0.5,0) rectangle (1,-.5);
\draw (0,0) --(0,-1.75);
\draw (.5,0) --(.5,-1.75);
\draw (1,0) --(1,-1.75);
\draw (1.5,0) --(1.5,-1.75);
\draw (2,0) --(2,-1.75);
\draw (0,0) --(2.25,0);
\draw (0,-.5) --(2.25,-.5);
\draw (0,-1) --(2.25,-1);
\draw (0,-1.5) --(2.25,-1.5);
\end{tikzpicture}
\end{center}
where one element is on the same row {\em and\/} column as other elements of $S$.
In this section we indicate how this case may be treated by the same technique used
in~\S\ref{singlerow} and~\S\ref{blocks}. Theorem~\ref{thm:Sigblocks} can then be
used to obtain $d_{n,r,S}$ for $S$ consisting of blocks including this configuration as
well as those studied in~\S\ref{blocks}. In \S\ref{ss:22square} we will further extend
the discussion to blocks consisting of $2\times 2$ squares. 
Not surprisingly, as the complexity of $S$
increases so do the technical subtleties needed to evaluate the relevant Segre classes.

\subsection{}\label{ss:corner}
Let $S$ consist of the entries $(1,1)$, $(1,2)$, $(2,1)$, as in the above figure. We will
adopt the set-up in previous sections: $L_S\cong \Pbb^2$ is the span of the three
entries in $S$ as a subspace of $\Pbb^{n^2-1}$; $\hL_S=G(k,n)\times L_S$; and
$\htau_{n,k}\subseteq G(k,n)\times \Pbb^{n^2-1}$ denotes the standard resolution
of the locus $\tau_{n,k}=\sigma_{n,n-k}$ of matrices with bounded rank.
The main information is carried by the Grassmann class of $S$, a projection of the
Segre class of $\hL_S\cap \htau_{n,k}$
in $\hL_S$ (Theorem~\ref{thm:mainred}).

It is easy to provide a qualitative description of $\hL_S\cap \htau_{n,k}$. By definition
(\S\ref{ss:resolution}) this is the locus of pairs $(K,\varphi)\in \hL_S=G(k,n)\times L_S$
such that $K\subseteq \ker\varphi$. We may use coordinates $a_{11}$, $a_{12}$, $a_{21}$
for $L_S\cong \Pbb^2$, corresponding to entries in the matrix
\[
\varphi=\begin{pmatrix}
a_{11} & a_{12} & 0 & \dots  \\
a_{21} & 0 & 0 & \dots  \\
0 & 0 & 0 & \dots  \\
\vdots & \vdots & \vdots & \ddots
\end{pmatrix}\quad.
\]
There are then three possibilities for $\ker\varphi$:
\begin{itemize}
\item If $a_{12}a_{21}\ne 0$, then $\rk \varphi=2$, and $E:=\ker\varphi$ is the 
codimension~$2$ subspace defined by the vanishing of the first two components:
$x_1=x_2=0$;
\item If $a_{12}=0$, defining a line $L_1\subseteq L_S$, then $\rk\varphi=1$, and 
$F:=\ker\varphi$ is the hyperplane defined by $x_1=0$;
\item If $a_{21}=0$, defining a line $L_2\subseteq L_S$ then $\rk\varphi=1$, and 
$F_\varphi:=\ker\varphi$ is the hyperplane with equation $a_{11}x_1+a_{12}x_2=0$.
\end{itemize}
Correspondingly, $\hL_S\cap \htau_{n,k}$ must consist of three components:
\begin{itemize}
\item $X=G(k,E)\times L_S\cong G(k,n-2)\times \Pbb^2$; $\codim_X L_S=2k$.
\item $W_1=G(k,F)\times L_1\cong G(k,n-1)\times \Pbb^1$; $\codim_{W_1} L_S=k+1$.
\item $W_2$, a Grassmann bundle over $L_2\cong \Pbb^1$;
the fiber over $\varphi\in L_2$ is $G(k,F_\varphi)$; 
\newline $\codim_{W_2} L_S=k+1$.
\end{itemize}

\begin{lemma}\label{lem:LS}
$\bullet$ The scheme $\hL_S\cap \htau_{n,k}$ consists of the reduced union of $X$,
$W_1$, $W_2$;

$\bullet$ The varieties $X$, $W_1$, $W_2$ are nonsingular. As classes in 
$\hL_S\cong G(k,n)\times \Pbb^2$,
\begin{align*}
s(X,\hL_S)^{\cO(-H)} &=\frac{c_k(\cSd)^2}{(1-H)\, c(\cSd\otimes\cO(-H))^2} \\
s(W_1,\hL_S)^{\cO(-H)} &= \frac{H\cdot c_k(\cSd)}{(1-H)\, c(\cSd\otimes\cO(-H))} \\
s(W_2,\hL_S)^{\cO(-H)} &= \frac{H\cdot c_k(\cSd\otimes \cO(H))}{(1-H)\, c(\cSd)}\quad.
\end{align*}
where $H$ is the pull-back of the hyperplane class from $\Pbb^2$, and $\cS$ denotes
the (pull-back of the) universal subbundle from $G(k,n)$.
\end{lemma}

\begin{proof}
The first statement may be verified by a coordinate computation analogous to the one
performed in the proof of~Lemma~\ref{lem:onerowdes}. Working in local coordinates
for $\hL_S$, e.g., in the chart
\[
\left(
C=\begin{pmatrix}
c_{11} & \dots & c_{1,n-k} & 1 & \dots & 0\\
\vdots & \ddots & \vdots & \vdots & \ddots & \vdots \\
c_{k1} & \dots & c_{k,n-k} & 0 & \dots & 1
\end{pmatrix}
, 
A=\begin{pmatrix}
1 & a_{12} & 0 & \dots  \\
a_{21} & 0 & 0 & \dots  \\
0 & 0 & 0 & \dots  \\
\vdots &  \vdots & \vdots & \ddots 
\end{pmatrix}
\right)
\]
equations for $\hL_S\cap \htau_{k,n}$ are $A\cdot C^t=0$, i.e., generators for the ideal of 
$\hL_S\cap \htau_{k,n}$ in $\hL_S$ in these coordinates are
\[
c_{11}+a_{12}c_{12} \,,\, \dots \,,\, c_{k1}+a_{12}c_{k2}\,,\,
a_{21} c_{11}\,,\,\dots\,,\, a_{21} c_{k1}\quad.
\]
The first point follows at once.

The second point follows from the description of $X$, $W_1$, $W_2$ preceding
the statement, from Lemma~\ref{lem:twifacts} \eqref{pt1}, and from computations of 
normal bundles analogous to those performed in the proof of Corollary~\ref{cor:onerowclass}.
\end{proof}

\begin{remark}
The scheme $\hL_S\cap \htau_{n,k}$ may in fact be described as the zero-scheme
of a section of $(\cSd\oplus \cSd)\otimes \cO(H)$.
\qede\end{remark}

\subsection{}
The fact that the information collected in~Lemma~\ref{lem:LS} is {\em not\/} enough in 
itself to compute $s(\hL_S\cap \htau_{n,k},\hL_S)^{\cO(-H)}$ is a good illustration of the 
subtleties involved in the notion of Segre class; the result that follows involves substantially 
more technical considerations. The classes presented in Lemma~\ref{lem:LS} should at least 
serve to make the general shape of the formula \eqref{eq:segreLS} given below look plausible.
For example, since both the classes for $W_1$ and $W_2$ are multiples of $H$, one should
expect the part of  $s(\hL_S\cap \htau_{n,k},\hL_S)^{\cO(-H)}$ that is {\em not\/} a multiple
of $H$ to agree with $s(X,\hL_S)^{\cO(-H)}$ modulo $H$:
\[
s(\hL_S\cap \htau_{n,k},\hL_S)^{\cO(-H)} = \frac{c_k(\cSd)^2}{c(\cSd)^2}
+ H \bigg(\cdots \bigg)\quad.
\]
On the other hand, since the top-dimensional components of $\hL_S\cap \htau_{n,k}$
are $W_1$ and $W_2$, one should expect the dominant terms in the Segre class to
be
\begin{align*}
H\cdot\left(c_k(\cSd)+c_k(\cSd\otimes \cO(H))\right)
&=H\cdot \left(c_k(\cSd) + c_k(\cSd) + H\cdot c_{k-1}(\cSd)+\cdots\right)
\\ &=2 c_k(\cSd)H + c_{k-1}(\cSd)H^2
\end{align*}
(note that $H^3=0$ since $H$ is the pull-back of the hyperplane in $\Pbb^2$). These 
features are indeed present in~\eqref{eq:segreLS}.

\begin{prop}
\begin{equation}\label{eq:segreLS}
s(\hL_S\cap \htau_{n,k},\hL_S)^{\cO(-H)} = \frac{c_k^2}{c^2}+\left(
\frac{2c_k}{c} - \frac{c_k^2}{c^2} 
\right)H
 + \left( \frac{k\,c_k+c_{k-1}}{c}
+\frac{c_k^2-\sum_{i=0}^k i\,c_ic_k}
{c^2}\right)H^2
\end{equation}
where $c=c(\cSd)$, $c_i=c_i(\cSd)$.
\end{prop}

\begin{proof}
We may blow-up $\hL_S$ along $W_1\cup W_2$; this produces a variety $\Til L$
with an exceptional divisor (consisting of two components), and one may verify that the 
(scheme-theoretic) inverse image of $\hL_S\cap \htau_{n,k}$ in $\Til L$ consists of the 
exceptional divisor and of a residual scheme~$\Til X$ that may be verified to equal 
the proper transform of $X$.
Applying Fulton's residual intersection formula (Proposition~9.2 in~\cite{85k:14004},
in the form given in~\cite{MR96d:14004}, Proposition~3) and Lemma~\ref{lem:twifacts} 
\eqref{pt2} one obtains the stated formula. The notationally rather demanding details 
are left to the reader.
\end{proof}

\begin{corol}\label{cor:SigLS}
For $S$ as above, and $k=n-r$,
\[
\Sigma_{n,r,S} =  \frac{k\,c_k+c_{k-1}}{c}+\frac{c_k^2-\sum_{i=0}^k i\,c_ic_k}{c^2}
\]
where $c=c(\cSd)$, $c_i=c_i(\cSd)$.
\end{corol}

\subsection{}
Theorem~\ref{thm:mainred} now yields a computation of the degrees:

\begin{theorem}\label{thm:LSdegs}
For $S=\{(1,1),(1,2),(2,1)\}$, the degrees $d_{n,r,S}$ of the projections of the rank
loci are given by
\begin{equation}\label{eq:LSdegs}
\int_{G(k,n)} c^{n-2} 
\left(c^2-(k c_k+c_{k-1})c
-(c_k^2-\sum_{i=0}^k i c_i c_k)\right)
\end{equation}
with $k=n-r$, and where $c=c(\cSd)$, $c_i=c_i(\cSd)$.
\end{theorem}

\begin{example}
For $n=7$:
\[
\begin{tikzpicture}[baseline=(current  bounding  box.center)]
\path [fill=red] (0,0) rectangle (.5,-.5);
\path [fill=red] (0,-0.5) rectangle (.5,-1);
\path [fill=red] (0.5,0) rectangle (1,-.5);
\draw (0,0) --(0,-3.5);
\draw (.5,0) --(.5,-3.5);
\draw (1,0) --(1,-3.5);
\draw (1.5,0) --(1.5,-3.5);
\draw (2,0) --(2,-3.5);
\draw (2.5,0) --(2.5,-3.5);
\draw (3,0) --(3,-3.5);
\draw (3.5,0) --(3.5,-3.5);
\draw (0,0) --(3.5,0);
\draw (0,-.5) --(3.5,-.5);
\draw (0,-1) --(3.5,-1);
\draw (0,-1.5) --(3.5,-1.5);
\draw (0,-2) --(3.5,-2);
\draw (0,-2.5) --(3.5,-2.5);
\draw (0,-3) --(3.5,-3);
\draw (0,-3.5) --(3.5,-3.5);
\end{tikzpicture}
\quad\quad\quad\quad
d_{7,r,S}=\begin{cases}
912 & r=1 \\
17303 & r=2 \\
15218 & r=3 \\
1001 & r=4 \\
6 & r=5 \\
0 & r=6 \\
0 & r=7
\end{cases}
\]
(Compiling this table took $0.06$ seconds on our computing equipment, {\it vs.}~about 15 
minutes by ordinary degree calculations. The case $n=8$ takes $0.13$ seconds applying 
Theorem~\ref{thm:LSdegs}, giving degrees $3418, 217007, 592956, 118188, 2548, 7, 0, 0$
for $r=1,\dots, 8$. This is already beyond the reach of a straightforward implementation
of a more direct method, at least on equipment available to us.)
\qede\end{example}

\begin{remark}\label{rem:onecor}
It appears that $d_{n,n-2,S}=n-1$ (as can probably be readily verified using~\eqref{eq:LSdegs}),
while
\[
d_{n,n-3,S} = 1,14,84,330,1001,2548,5712,\dots	\quad\quad(n\ge 3)
\]
are the dimensions $\dim V^{(n-3)}$ appearing in the case $a=1$ of Theorem~7.2 
in~\cite{MR2204753}, on Lie algebras in the subexceptional series. These numbers also
arise in enumerating paths in the plane (cf.~\cite{oeis}), and as Kekul\'e numbers for benzenoid 
hydrocarbons configured as prolate pentagons $D^i(3,n-2)$, \cite{benzenoids1988kekule}, 
p.~183.
\qede\end{remark}

\subsection{}
Using Theorem~\ref{thm:Sigblocks}, we can now compute the degrees $d_{n,r,S}$ for 
$S$ obtained by including blocks with the corner shape analyzed above. For example, 
$S$ could consist of several such `corners', provided that rows and columns of distinct 
configurations do not overlap.
\begin{center}
\begin{tikzpicture}
\path [fill=red] (0,0) rectangle (.5,-.5);
\path [fill=red] (0,-0.5) rectangle (.5,-1);
\path [fill=red] (0.5,0) rectangle (1,-.5);
\path [fill=red] (1,-1) rectangle (1.5,-1.5);
\path [fill=red] (1,-1.5) rectangle (1.5,-2);
\path [fill=red] (1.5,-1) rectangle (2,-1.5);
\path [fill=red] (2,-2) rectangle (2.5,-2.5);
\path [fill=red] (2,-2.5) rectangle (2.5,-3);
\path [fill=red] (2.5,-2) rectangle (3,-2.5);
\draw (0,0) --(0,-3.25);
\draw (.5,0) --(.5,-3.25);
\draw (1,0) --(1,-3.25);
\draw (1.5,0) --(1.5,-3.25);
\draw (2,0) --(2,-3.25);
\draw (2.5,0) --(2.5,-3.25);
\draw (3,0) --(3,-3.25);
\draw (0,0) --(3.25,0);
\draw (0,-.5) --(3.25,-.5);
\draw (0,-1) --(3.25,-1);
\draw (0,-1.5) --(3.25,-1.5);
\draw (0,-2) --(3.25,-2);
\draw (0,-2.5) --(3.25,-2.5);
\draw (0,-3) --(3.25,-3);
\end{tikzpicture}
\end{center}

\begin{theorem}\label{thm:manycorners}
Suppose $S=\cup_{i=1}^g \{(a_i,b_i),(a_i,b_i+1),(a_i+1,b_i)\}$, where the different groups
have no overlapping rows or columns. Then
\begin{equation*}
d_{n,r,S} = 
\int_{G(k,n)} c^{n-2g} 
\left(c^2-(k c_k+c_{k-1})c
-(c_k^2-\sum_{i=0}^k i c_i c_k)\right)^g
\end{equation*}
with $k=n-r$, and where $c=c(\cSd)$, $c_i=c_i(\cSd)$.
\end{theorem}

\begin{proof}
This follows from Theorem~\ref{thm:Sigblocks} and Corollary~\ref{cor:SigLS}.
\end{proof}

\begin{example}
For $n=7,g=3$, and after a permutation of the rows of the matrix:
\[
\begin{tikzpicture}[baseline=(current  bounding  box.center)]
\path [fill=red] (0,0) rectangle (.5,-.5);
\path [fill=red] (0,-.5) rectangle (.5,-1);
\path [fill=red] (.5,-.5) rectangle (1,-1);
\path [fill=red] (1,-1) rectangle (1.5,-1.5);
\path [fill=red] (1,-1.5) rectangle (1.5,-2);
\path [fill=red] (1.5,-1.5) rectangle (2,-2);
\path [fill=red] (2,-2) rectangle (2.5,-2.5);
\path [fill=red] (2,-2.5) rectangle (2.5,-3);
\path [fill=red] (2.5,-2.5) rectangle (3,-3);
\draw (0,0) --(0,-3.5);
\draw (.5,0) --(.5,-3.5);
\draw (1,0) --(1,-3.5);
\draw (1.5,0) --(1.5,-3.5);
\draw (2,0) --(2,-3.5);
\draw (2.5,0) --(2.5,-3.5);
\draw (3,0) --(3,-3.5);
\draw (3.5,0) --(3.5,-3.5);
\draw (0,0) --(3.5,0);
\draw (0,-.5) --(3.5,-.5);
\draw (0,-1) --(3.5,-1);
\draw (0,-1.5) --(3.5,-1.5);
\draw (0,-2) --(3.5,-2);
\draw (0,-2.5) --(3.5,-2.5);
\draw (0,-3) --(3.5,-3);
\draw (0,-3.5) --(3.5,-3.5);
\end{tikzpicture}
\quad\quad\quad\quad
d_{7,r,S}=\begin{cases}
888 & r=1 \\
13395 & r=2 \\
4078 & r=3 \\
2 & r=4 \\
0 & r=5 \\
0 & r=6 \\
0 & r=7
\end{cases}
\]
The $r=2$ case was mentioned in the introduction, and the $r=4$ case in Remark~\ref{rem:dto1}.
\qede\end{example}

\begin{example}
By the same token, one can include all blocks for which we have computed Grassmann
classes. For instance:
\[
\begin{tikzpicture}[baseline=(current  bounding  box.center)]
\path [fill=red] (0,0) rectangle (.5,-.5);
\path [fill=red] (0,-.5) rectangle (.5,-1);
\path [fill=red] (.5,-.5) rectangle (1,-1);
\path [fill=red] (1,-1) rectangle (1.5,-2);
\path [fill=red] (1.5,-2) rectangle (2.5,-2.5);
\path [fill=red] (2.5,-2.5) rectangle (3,-3);
\path [fill=red] (2.5,-3) rectangle (3,-3.5);
\path [fill=red] (3,-3) rectangle (3.5,-3.5);
\draw (0,0) --(0,-3.5);
\draw (.5,0) --(.5,-3.5);
\draw (1,0) --(1,-3.5);
\draw (1.5,0) --(1.5,-3.5);
\draw (2,0) --(2,-3.5);
\draw (2.5,0) --(2.5,-3.5);
\draw (3,0) --(3,-3.5);
\draw (3.5,0) --(3.5,-3.5);
\draw (0,0) --(3.5,0);
\draw (0,-.5) --(3.5,-.5);
\draw (0,-1) --(3.5,-1);
\draw (0,-1.5) --(3.5,-1.5);
\draw (0,-2) --(3.5,-2);
\draw (0,-2.5) --(3.5,-2.5);
\draw (0,-3) --(3.5,-3);
\draw (0,-3.5) --(3.5,-3.5);
\end{tikzpicture}
\quad\quad\quad\quad
d_{7,r,S}=\begin{cases}
886 & r=1 \\
12967 & r=2 \\
3102 & r=3 \\
0 & r=4 \\
0 & r=5 \\
0 & r=6 \\
0 & r=7
\end{cases}
\]
The degrees are obtained by applying Theorem~\ref{thm:Sigblocks} together with
the computation of the relevant classes $\Sigma_{n,r,S}$ in Corollary~\ref{cor:SigmaSrow},
\ref{cor:SigmaScol}, and~\ref{cor:SigLS}.
\qede\end{example}

\subsection{}\label{ss:22square}
As a last example we consider a set $S$ consisting of entries in a $2\times 2$ square:
for instance, $S=\{(1,1),(1,2),(2,1),(2,2)\}$.
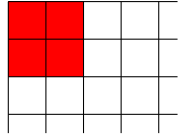
\begin{center}
\begin{tikzpicture}
\path [fill=red] (0,0) rectangle (.5,-.5);
\path [fill=red] (0,-0.5) rectangle (.5,-1);
\path [fill=red] (0.5,0) rectangle (1,-.5);
\path [fill=red] (0.5,-0.5) rectangle (1,-1);
\draw (0,0) --(0,-1.75);
\draw (.5,0) --(.5,-1.75);
\draw (1,0) --(1,-1.75);
\draw (1.5,0) --(1.5,-1.75);
\draw (2,0) --(2,-1.75);
\draw (0,0) --(2.25,0);
\draw (0,-.5) --(2.25,-.5);
\draw (0,-1) --(2.25,-1);
\draw (0,-1.5) --(2.25,-1.5);
\end{tikzpicture}
\end{center}
The determination of $\hL_S\cap \htau_{n,k}$ as a subscheme of $\hL_S=G(k,n)\times \Pbb^3$
is very similar to the one reviewed for the case considered 
in \S\ref{ss:corner}. This locus consists of two components: a codimension~$2k$ component 
dominating $\Pbb^3$, and a codimension~$k+1$ component dominating the determinantal
quadric $Q$ given by $a_{11}a_{22}-a_{12}a_{21}=0$. The Segre class may be computed 
by blowing up this latter locus and using residual intersection and Lemma~\ref{lem:twifacts} 
\eqref{pt2}, as in~\S\ref{ss:corner}. The computation yields the class
\begin{multline*}
s(\hL_S\cap \htau_{n,k},\hL_S)^{\cO(-H)}
=\frac {[G(k,n-2)\times \Pbb^3]}{(1+H) c(\cSd\otimes \cO(H))^2} 
+\frac{c_k(\cSd\otimes \cO(h_1))[G(k,n)\times Q]}{(1-H)(1+H)c(\cSd\otimes \cO(-h_2))}\\
+\frac{\left( \sum_{j=0}^2 \sum_{i=0}^k \binom i{j+1} (1-H)^{i-j} c_{k-i}(\cSd\otimes \cO(h_1))
(2H)^j\right)[G(k,n-2)\times Q]}{(1-H)(1+H) c(\cSd\otimes \cO(-h_2))
c(\cSd\otimes \cO(H))^2}\quad,
\end{multline*}
where $H$ denotes the pull-back of the hyperplane class in $\Pbb^3$ and of its restriction
to $Q$, and $h_1$, $h_2$ are pull-backs of the classes from the two rulings on $Q$.

The Grassmann class of a $2\times 2$ square can be computed from this expression:
\begin{multline*}
\Sigma_{n,r,S} = -\sum_{i=0}^r \sum_{j=0}^r \binom{2k+i+j+3}3
	c_i(\cQd) c_j(\cQd) c_k(\cSd)^2 
+\sum_{j=0}^r c_j(\cQd)\big(2 c_k(\cSd) \\
+(k+j) c_{k-1}(\cSd)\big) 
+\sum_{i=0}^k \sum_{u=0}^r \sum_{v=0}^r \sum_{w=0}^r
\left( 
2\binom i3 - 2(k-u+v+w+1)\binom i2 \right. \\
\left. + (k-u+v+w+2)(2k+v+w+1)\binom i1
\right) c_u(\cQd) c_v(\cQd) c_w(\cQd) c_{k-i}(\cSd)c_k(\cSd)^2\quad.
\end{multline*}

Theorem~\ref{thm:mainred} can then be used to compute $d_{n,r,S}$ for sets $S$
including blocks consisting of $2\times 2$ squares.

\begin{example}\label{ex:squareexs}
For $n=7$:
\[
\begin{tikzpicture}[baseline=(current  bounding  box.center)]
\path [fill=red] (0,0) rectangle (.5,-.5);
\path [fill=red] (0,-0.5) rectangle (.5,-1);
\path [fill=red] (0.5,0) rectangle (1,-.5);
\path [fill=red] (0.5,-0.5) rectangle (1,-1);
\draw (0,0) --(0,-3.5);
\draw (.5,0) --(.5,-3.5);
\draw (1,0) --(1,-3.5);
\draw (1.5,0) --(1.5,-3.5);
\draw (2,0) --(2,-3.5);
\draw (2.5,0) --(2.5,-3.5);
\draw (3,0) --(3,-3.5);
\draw (3.5,0) --(3.5,-3.5);
\draw (0,0) --(3.5,0);
\draw (0,-.5) --(3.5,-.5);
\draw (0,-1) --(3.5,-1);
\draw (0,-1.5) --(3.5,-1.5);
\draw (0,-2) --(3.5,-2);
\draw (0,-2.5) --(3.5,-2.5);
\draw (0,-3) --(3.5,-3);
\draw (0,-3.5) --(3.5,-3.5);
\end{tikzpicture}
\quad\quad\quad\quad
d_{7,r,S}=\begin{cases}
887 & r=1 \\
14701 & r=2 \\
9478 & r=3 \\
371 & r=4 \\
1 & r=5 \\
0 & r=6 \\
0 & r=7
\end{cases}
\]
\[
\begin{tikzpicture}[baseline=(current  bounding  box.center)]
\path [fill=red] (0,0) rectangle (.5,-1);
\path [fill=red] (.5,-1) rectangle (1.5,-1.5);
\path [fill=red] (1.5,-1.5) rectangle (2,-2.5);
\path [fill=red] (2,-2) rectangle (2.5,-2.5);
\path [fill=red] (2.5,-2.5) rectangle (3.5,-3.5);
\draw (0,0) --(0,-3.5);
\draw (.5,0) --(.5,-3.5);
\draw (1,0) --(1,-3.5);
\draw (1.5,0) --(1.5,-3.5);
\draw (2,0) --(2,-3.5);
\draw (2.5,0) --(2.5,-3.5);
\draw (3,0) --(3,-3.5);
\draw (3.5,0) --(3.5,-3.5);
\draw (0,0) --(3.5,0);
\draw (0,-.5) --(3.5,-.5);
\draw (0,-1) --(3.5,-1);
\draw (0,-1.5) --(3.5,-1.5);
\draw (0,-2) --(3.5,-2);
\draw (0,-2.5) --(3.5,-2.5);
\draw (0,-3) --(3.5,-3);
\draw (0,-3.5) --(3.5,-3.5);
\end{tikzpicture}
\quad\quad\quad\quad
d_{7,r,S}=\begin{cases}
861 & r=1 \\
10701 & r=2 \\
1424 & r=3 \\
0 & r=4 \\
0 & r=5 \\
0 & r=6 \\
0 & r=7
\end{cases}
\]
For $S$ consisting of a single $2\times 2$ square, the degrees $d_{n,n-3,S}$:
\[
d_{n,n-3,S} = 1,10,46,146,371,812,1596,2892,\dots \quad\quad(n\ge 3)
\]
agree with the Kekul\'e numbers for benzenoid hydrocarbons in the symmetric `chevron' 
configuration $Ch(3,n-3)$, cf.~\cite{benzenoids1988kekule}, p.~120.
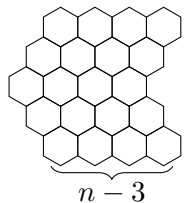
\begin{center}
\begin{tikzpicture}
  \begin{scope}[%
every node/.style={anchor=west,
regular polygon,
regular polygon rotate = 30, 
regular polygon sides=6,
draw,
minimum width=15pt,
outer sep=0,
},
      transform shape]
    \node (A) {};
    \node (B) at (A.side 5) {};
    \node (C) at (B.side 5) {};
    \node (D) at (C.side 5) {};
  \end{scope}
  \begin{scope}[%
every node/.style={anchor=north,
regular polygon,
regular polygon rotate = 30, 
regular polygon sides=6,
draw,
minimum width=15pt,
outer sep=0,
},
      transform shape]
    \node ({EA}) at (A.corner 3) {};
    \node ({FA}) at ({EA}.corner 3) {};
    \node ({GA}) at ({FA}.corner 5) {};
    \node ({HA}) at ({GA}.corner 5) {};
    \node ({EB}) at (B.corner 3) {};
    \node ({FB}) at ({EB}.corner 3) {};
    \node ({GB}) at ({FB}.corner 5) {};
    \node ({HB}) at ({GB}.corner 5) {};
    \node ({EC}) at (C.corner 3) {};
    \node ({FC}) at ({EC}.corner 3) {};
    \node ({GC}) at ({FC}.corner 5) {};
    \node ({HC}) at ({GC}.corner 5) {};
    \node ({ED}) at (D.corner 3) {};
    \node ({FD}) at ({ED}.corner 3) {};
    \node ({GD}) at ({FD}.corner 5) {};
    \node ({HD}) at ({GD}.corner 5) {};
  \end{scope}
	\draw [decorate,decoration={brace,amplitude=5pt},
			xshift=0pt, yshift=0pt]
			(1.72,-1.85) -- (0.1,-1.85)
			node[black,midway,yshift=-10pt]
				{$n-3$};
\end{tikzpicture}
\end{center}
For $S$ consisting of {\em three\/} blocks of $2\times 2$ squares, the numbers
\[
d_{n,n-4,S} = 105, 336, 825, 1716, 3185, 5440, \dots\quad\quad (n\ge 6)
\]
agree with $d_{2n+1,2n-2|2|3}$, cf.~\S\ref{ss:benzene2}, and are in particular
Kekul\'e numbers for hexagonal configurations of benzenoids.
\qede\end{example}


\newcommand{\etalchar}[1]{$^{#1}$}

\end{document}